\documentclass[11pt]{amsart} 
\usepackage{amsthm,amsbsy,amsfonts,amssymb,amscd}
\usepackage{bbm}
\usepackage[enableskew]{youngtab}
\usepackage[mathscr]{euscript} 

\usepackage{enumerate} 
\usepackage{tikz} 
\usetikzlibrary{calc,intersections,through, backgrounds,arrows,decorations.pathmorphing,backgrounds,positioning,fit,petri} 
\usetikzlibrary{calc}

\usepackage{tikz-cd} 
\usepackage[all]{xy}
\xyoption{knot}
\xyoption{poly}

\usepackage{hyperref}

\usepackage[margin=1.1in]{geometry}

\newcommand{\Fn}{s\hspace{-0.3mm}\cF_n}

\newcommand{\fcs}{\mathsf{fcs}}

\newcommand{\pe}{\mathfrak{pe}}

\newcommand{\DA}{\mathcal{P}\hspace{-0.25mm}\mathcal{D}}

\newcommand{\LL}{{\mathbf{L}}}

\newcommand{\ideal}{\mathscr{I}}
\newcommand{\KK}{\mathscr{K}}

\newcommand{\sve}{\mathsf{svec}_{\mk}}

\newcommand{\cF}{{\mathcal{F}}}

\newcommand{\leeg}{\mbox{{\color{white} Test}}}

\newcommand{\Par}{{\mathsf{Par}}}

\newcommand{\bb}{\mathsf{b}}
\newcommand{\cc}{\mathsf{c}}
\newcommand{\aac}{\mathsf{a}}

\newcommand{\PA}{\mathsf{Par}}

\newcommand{\TL}{{\mathrm{TL}}} 

\newcommand{\II}{\mathrm{I}}
\newcommand{\XX}{\mathrm{X}}

\newcommand{\tto}{\twoheadrightarrow}

\newcommand{\cB}{\mathcal B}

\newcommand{\im}{{\rm{im}\,}}

\newcommand{\Id}{\mathrm{Id}}

\newcommand{\Hom}{\mathrm{Hom}}

\newcommand{\unit}{{\mathbbm{1}}}

\newcommand{\minus}{\scalebox{0.9}{{\rm -}}}

\newcommand{\TT}{\mathbf{T}}
\newcommand{\Nat}{\mathrm{Nat}} 
\newcommand{\End}{\mathrm{End}}

\newcommand{\dfs}{{/\kern-2pt/}}

\newcommand{\JJJ}{\mathscr{J}}
\newcommand{\op}{{\rm op}}
\newcommand{\mk}{\Bbbk}
\newcommand{\cC}{\mathcal{C}}

\newcommand{\cS}{\mathcal{S}}
\newcommand{\cA}{\mathcal{A}}
\newcommand{\Ob}{{\rm{Ob}\,}}

\newcommand{\charr}{{\rm{char}}}

\newcommand{\fg}{\mathfrak{g}}

\newcommand{\mN}{\mathbb{N}}

\newcommand{\mZ}{\mathbb{Z}}
\newcommand{\mF}{\mathbb{F}}
\newcommand{\mS}{\mathbb{S}}

\numberwithin{equation}{section}

\newcommand{\mP}{{\mathfrak{P}}}

\newcommand{\oa}{\bar{0}}

\newcommand{\ob}{\bar{1}}

\newtheoremstyle{notes} {} {} {} {} {\bfseries} {.} {.5em} {}

\theoremstyle{plain}

\newtheorem{prop}[subsubsection]{Proposition}

\newtheorem{lemma}[subsubsection]{Lemma}

\newtheorem{cor}[subsubsection]{Corollary}

\newtheorem{thm}[subsubsection]{Theorem}

\theoremstyle{remark}

\newtheorem{rem}[subsubsection]{Remark} 

\newtheorem{ddef}[subsubsection]{Definition} 

\swapnumbers

\usepackage{etoolbox}

\usepackage{ytableau}
\usepackage{youngtab}

\makeatletter
\pretocmd{\appendix}{\addtocontents{toc}{\protect\addvspace{10\p@}}}{}{}
\makeatother

\theoremstyle{remark}

\newtheorem{ex}[subsubsection]{Example}

\newtheoremstyle{construction} {} {} {} {} {\bfseries} { } {0pt} {}

\theoremstyle{construction}

\title[The periplectic Deligne category]{The periplectic Brauer algebra III: The Deligne category}

\newcommand{\ind}{{\rm Ind}}
\newcommand{\res}{{\rm Res}}



\begin{document} 
\date{} 
\begin{abstract}
We construct a faithful categorical representation of an infinite Temperley-Lieb algebra on the periplectic analogue of Deligne's category.
We use the corresponding combinatorics to classify thick tensor ideals in this periplectic Deligne category.
This allows to determine the objects in the kernel of the monoidal functor going to the module category of the periplectic Lie supergroup.
We use this to classify indecomposable direct summands in the tensor powers of the natural representation, determine which are projective and determine their simple top.
	\end{abstract}

	\author{Kevin Coulembier}
\address{K.C.: School of Mathematics \& Statistics, University of Sydney, NSW 2006, Australia}
\email{kevin.coulembier@sydney.edu.au}

\author{Michael Ehrig}
\address{M.E.: School of Mathematics \& Statistics, University of Sydney, NSW 2006, Australia}
\email{michael.ehrig@sydney.edu.au}

\maketitle 


\section*{Introduction} 

This is the third paper in a series studying an analogue of the Brauer algebra which appears in invariant theory for the periplectic Lie superalgebra, see \cite{Moon}. In~\cite{PB1} the first author studied cellular and homological properties of the algebras over fields of arbitrary characteristic, leading in particular to a classification of the blocks in characteristic zero. In~\cite{PB2} we completed this by determining the Jordan-H\"older decomposition multiplicities of projective and cell modules.

In the current paper, we study the periplectic analogue $\DA$ of the {\em Deligne category} of~\cite{Deligne}, a strict monoidal supercategory with universal properties, defined in~\cite{Kujawa, Vera}. We construct a {\em categorical representation of~$\TL_\infty(0)$, the infinite Temperley-Lieb algebra with the circle evaluated at zero}, on~$\DA$. This can be interpreted as a natural analogue of the categorical representation of~$\mathfrak{sl}_{\infty}$ on module categories of symmetric groups or polynomial functors, see \cite{Oded, LLT}. Moreover, our approach should be adaptable to construct a categorical representation of~$\mathfrak{sl}_{\infty/2}\oplus \mathfrak{sl}_{\infty/2}$ on the ordinary Deligne category~$\underline{\rm Rep}(O_\delta)$ of \cite{Deligne, Comes}, which relates to~\cite{GS}.

Our categorical representation of~$\TL_\infty(0)$ is a `weak categorification' of a representation in the terminology of \cite{Maz}, since there is no known 2-categorical or monoidal notion of categorification of~$\TL_\infty$ that incorporates the specialisation at $0$. We prove that the representation we categorify, which is a representation of~$\TL_\infty(0)$ on bosonic Fock space, is faithful, which might be of use in developing such a notion.
The categorical representation of~$\TL_\infty(0)$ admits a filtration, where each composition factor corresponds to a cell of the monoidal supercategory $\DA$. Moreover, we show that the decategorification of the composition factors are isomorphic to representations of~$\TL_\infty(0)$ categorified in~\cite{gang}, and that both categorifications are very closely related.


The functor on the Deligne category which lies at the basis of the categorical representation is the tensor product with the generator. Its combinatorics determines explicitly the structure of the tensor product of this  generator and an arbitrary indecomposable object. In particular we use this to {\em classify the thick tensor ideals and cells} in the periplectic Deligne category $\DA$. The corresponding classification for the Deligne category~$\underline{\rm Rep}(O_\delta)$ was obtained in~\cite{Comes}. We use a different approach, compared to \cite{Comes}, to prove that the combinatorics of the tensor functor is related to the decomposition multiplicities of the periplectic Brauer algebra in \cite{PB2}. This approach is much more direct, since it does not rely on liftings of idempotents or classical invariant theory, and can thus be applied in many similar situations (including the one in~\cite{Comes}). In subsequent work in \cite{Ideals}, the first author will prove that our classification of thick tensor ideals on the level of objects actually yields a complete classification of the tensor ideals in~$\DA$ on the level of morphisms as well.

There exists a tensor functor from the periplectic Deligne category to the category of finite dimensional modules over the periplectic Lie supergroup, see \cite{Kujawa, Vera}, which is full by results in~\cite{DLZ}. Its kernel must be a thick tensor ideal and similarly the pre-image of the class of projective modules is a thick tensor ideal. Our classification of thick tensor ideals allows to determine efficiently those ideals. This thus yields {\em a classification of the indecomposable direct summands in the tensor powers of the natural representation} for the periplectic Lie supergroup. Furthermore, we determine {\em which direct summands are projective}. These results are analogues of the corresponding ones for orthosymplectic Lie supergroups in~\cite{Comes}. In contrast to \cite{Comes}, our methods do not rely on cohomological tensor functors and instead use simple combinatorial considerations to deduce the classification. Finally, we also {\em describe explicitly the highest weight of the top of each projective cover in terms of the combinatorics of the Deligne category}.

The paper is organised as follows.  After recalling some definitions and introducing some notation concerning monoidal supercategories and periplectic Brauer algebras in Section~\ref{SecPrel}, we study the elementary properties of the periplectic Deligne category~$\DA$ in Section~\ref{SecDelCat}. In Section~\ref{SecTensT}, we study the functor~$\TT$ on the Deligne category which corresponds to taking the tensor product with the generator. We prove that its action on objects can be described in terms of the decomposition multiplicities of the periplectic Brauer algebra in \cite{PB2} and that it decomposes as $\TT=\oplus_{i\in\mZ}\TT_i$ according to the eigenvalues of a natural transformation. Section~\ref{SecTL} is a purely combinatorial section where we prove uniqueness and existence of a representation of the Temperley-Lieb algebra~$\TL_{\infty}(0)$ on the space of partitions (the Fock space). It then follows that the functors~$\TT_i$ decategorify to this representation. We also prove that the representation is faithful and establish a filtration.
Section~\ref{SecMain} contains our main results, the classification of thick tensor ideals in~$\DA$ and the description of the higher tensor powers of the natural representation of the periplectic Lie supergroup. Finally, in Section~\ref{SecCateg} we construct natural transformations related to the functors~$\TT_i$ in order to improve the above decategorification statements to an actual categorical representation and filtration. Furthermore, we establish a connection between the composition factors of the filtration of our categorical representation and the categorical representations of~$\TL_\infty(0)$ in~\cite{gang}.

\section{Preliminaries}\label{SecPrel}

We set~$\mN=\{0,1,2,\ldots\}$. For a given set $S$, the power set is denoted by~$\mP(S)$ and the set of subsets of cardinality $n$ by~$\mP(S;n)$. Throughout the paper, $\mk$ is an algebraically closed field of characteristic zero. Let~$\sve$ denote the monoidal category of all $\mF_2$-graded $\mk$-vector spaces, with grading preserving morphisms. For elements~$v$ of degree $\oa$, resp. $\ob$, in a graded vectorspace, we write $|v|=0$, resp. $|v|=1$. For any $r\in\mZ_{\ge 1}$, we introduce the sets
$$\JJJ(r):=\{r-2i\,|\, 0\le i\le r/2\}\qquad\mbox{and}\qquad\JJJ^0(r):=\{r-2i\,|\, 0\le i< r/2\}.$$
Furthermore, we set~$\JJJ(0)=0=\JJJ^0(0)$. 

\subsection{Partitions}

We denote the set of partitions of all numbers by~$\PA$. The free $\mZ$-module of~$\mZ$-linear combinations of the elements of~$\PA$ will be denoted by~$\PA_{\mZ}$. All matrices that will appear in the paper will have their columns and rows labelled by~$\PA$ and have entries in~$\mZ$. 

We will identify a partition with its Young diagram, using English notation. Each box or node in the diagram has coordinates $(i,j)$, meaning that the box is in row $i$ and column~$j$. The {\em content} of a box in position~$(i,j)$ in a Young diagram is $j-i$. Any box with content $q$ will be referred to as a~$q$-box. The value $i+j$ will be referred to as the {\em anticontent} of the box. 

By a {\em rim hook} of~$\lambda$ we mean a removable and connected hook of~$\lambda$. By a (rim) $a$-hook for $a \in \mathbb{N}$ we mean a rim hook with~$a$ boxes. In case~$\lambda$ admits an addable~$q$-box, we write the partition obtained by adding said box as $\lambda  \boxplus q$. In case~$\lambda$ has a removable~$q$-box, we write the partition obtained by removing said box as $\lambda  \boxminus q$. 

For~$k\in\mN$, we fix the partition~$\partial^k$ of~$\frac{1}{2}k(k+1)$, defined as
$$\partial^k:=(k,k-1,\ldots,1,0).$$
The set~$\{\partial^k\,|\,k\in\mN\}$ thus consists of all $2$-cores, {\it i.e.} all partitions from which one cannot remove any rim 2-hook. For all $k\in\mN$, we define the following subsets of~$\PA$:
\begin{equation}\label{SetsPAk}\PA^{\ge k}=\{\lambda\,|\,\partial^k \subseteq \lambda\},\quad \PA^{\le k}=\PA\setminus \PA^{\ge(k+1)}\;\,\mbox{and}\quad \PA^k=\PA^{\ge k}\cap\PA^{\le k}.\end{equation}

\subsection{Supercategories} We recall some definitions of \cite[Section~1]{supercategory}. 
\subsubsection{}
A {\em supercategory} is defined as a category enriched over $\sve$. {\em Superfunctors} between supercategories are functors enriched in the same way. By definition, supercategories and superfunctors are thus in particular $\mk$-linear. 

For two supercategories~$\cB$ and~$\cC$, the supercategory~$\cB\boxtimes \cC$ has as objects ordered pairs $(X,Y)$, with~$X\in\Ob\cB$ and~$Y\in\Ob\cC$, and morphism spaces given by
$$\Hom_{\cB\boxtimes \cC}((X_1,Y_1),(X_2,Y_2))\;=\;\Hom_{\cB}(X_1,X_2)\otimes_\mk \Hom_{ \cC}(Y_1,Y_2),$$
with composition defined by the super interchange law
\begin{equation}\label{superComp}(f\boxtimes g)\circ (h\boxtimes k)\;=\; (-1)^{|h||g|}(f\circ h)\boxtimes (g\circ k).\end{equation}

\subsubsection{Natural transformations}Consider two supercategories~$\cC_1,\cC_2$ and superfunctors~$F,G:\cC_1\to\cC_2$.
A {\em natural transformation of superfunctors~$\xi:F\Rightarrow G$ of parity $p\in\mF_2$} is a family $\{\xi_X:FX\to GX\,|\,X\in\Ob\cC_1\}$ of morphisms of parity $p$ such that for any homogeneous morphism $\alpha: X\to Y$ in~$\cC_1$, we have $G(\alpha)\circ\xi_X=(-1)^{p|\alpha|}\xi_{Y}\circ F(\alpha)$. An even natural transformation of superfunctors is thus just an ordinary natural transformation, where every morphism is even. All functors appearing will be superfunctors, thus all natural transformations appearing are considered as natural transformations of superfunctors.
The space $\Nat(F,G)$ of natural transformation of superfunctors $F\Rightarrow G$ is thus $\mF_2$-graded.

In the following three paragraphs we recall some standard manipulations of natural transformations. For ease of reading we leave out the categories on which the various functors are defined, as it should be clear from context.

For a functor~$F$ and a natural transformation~$\xi:G_1\Rightarrow G_2$, we denote by~$F(\xi): F\circ G_1\Rightarrow F\circ G_2$ the natural transformation given by~$F(\xi)_X=F(\xi_X)$. The natural transformation~$\xi_F:G_1\circ F\Rightarrow G_2\circ F$ is defined as $(\xi_F)_X=\xi_{FX}$. 

For two natural transformations $\xi_1:F_1\Rightarrow G_1$ and~$\xi_2: F_2\Rightarrow G_2$, we denote the horizontal composition, or \emph{Godement product}, by~$\xi_1\star \xi_2: F_1\circ F_2\Rightarrow G_1\circ G_2$, which is the natural transformation~$G_1(\xi_2)\circ(\xi_{1})_{F_2}=(\xi_1)_{G_2}\circ F_1(\xi_2)$.

For two natural transformations $\xi_1:F\Rightarrow G$ and~$\xi_2:G\Rightarrow H$, we denote the vertical composition by~$\xi_2\circ\xi_1:F\Rightarrow H$, this is the natural transformation defined by~$(\xi_2\circ\xi_1)_X=(\xi_2)_X\circ (\xi_1)_X$.

\subsubsection{$\Ob$-kernel of a functor} We say that a functor is {\em essentially surjective} if any object in the target category is isomorphic to one in the image. The Ob-kernel of a functor is the full subcategory of the source category of all objects which are sent to zero. A functor is {\em essentially injective} if it has trivial Ob-kernel. A functor is {\em essentially bijective} if it is both essentially injective and surjective.

\subsubsection{Monoidal supercategories}
A {\em strict monoidal supercategory} is a supercategory~$\cC$ equipped with a superfunctor 
$\cC\boxtimes\cC\to\cC$ denoted by~$-\otimes-$, and a unit object ${\mathbbm{1}}_{\cC}={\mathbbm{1}}$, such that we have equalities of functors~$\unit\otimes-=\Id=-\otimes\unit$ and
 $(-\otimes-)\otimes -=-\otimes(-\otimes -)$. When we omit `strict', these equalities are replaced by three even natural isomorphisms, satisfying the ordinary (since they are all even) coherence conditions, {\it i.e.} the commuting pentagon and triangle diagram.
 
 A {\em braiding} $B$ in a monoidal supercategory is a family of even isomorphisms in~$\cC$
 $$\{B_{X,Y}:\;X\otimes Y\;\to\; Y\otimes X\,|\,X,Y\in\Ob\cC\},$$
 such that~$B_{X',Y'}\circ (f\otimes g)=(-1)^{|f||g|}(g\otimes f)\circ B_{X,Y}$ for any two morphisms $f:X\to X'$ and~$g:Y\to Y'$ and the usual commutative diagrams for~$B_{X,Y\otimes Z}$ and~$B_{X\otimes Y,Z}$ hold true. If $B_{X,Y}\circ B_{Y,X}=1_{Y\otimes X}$ the braiding is {\em symmetric}.

 
 \subsubsection{}For two monoidal supercategories~$\cC_1$ and~$\cC_2$, a {\em monoidal superfunctor} is a superfunctor~$F:\cC_1\to \cC_2$ with an even natural isomorphism $c: (F-)\otimes (F-)\Rightarrow F\circ(-\otimes-)$ and an even isomorphism $i:\unit_{\cC_2}\to F(\unit_{\cC_1})$ satisfying the ordinary (because again all morphisms are even) commuting diagrams with the natural isomorphisms of the monoidal structure on~$\cC_1$ and $\cC_2$.

\subsection{The periplectic Brauer category}\label{ACat}
The {\em periplectic Brauer category} $\cA$, was introduced as the category~$\cB(0,-1)$ in~\cite{Kujawa}, see also~\cite{Vera, supercategory, PB1}. It is a small skeletal supercategory with~$\Ob\cA=\mN$. Note that in~\cite{Kujawa}, contravariant composition of morphisms is used, contrary to \cite{Vera, supercategory, PB1}. We thus actually have $\cA=\cB(0,-1)^{\op}$.

\subsubsection{Brauer diagrams}\label{SSBrdiag}
The vector space $\Hom_{\cA}(i,j)$ is zero unless~$i+j$ is even. Furthermore, the graded vectorspace $\Hom_{\cA}(i,j)$ is purely even, resp. purely odd, if $(i-j)/2$ is even, resp. odd. 
The vector space $\Hom_{\cA}(i,j)$ is spanned by {\em $(i,j)$-Brauer diagrams}. These diagrams correspond to all partitions of a set of~$i+j$ dots into pairs. They are graphically represented by~$i$ dots on a horizontal line and~$j$ dots on a second horizontal line, above the first one. The Brauer diagram then consists of~$(i+j)/2$ lines, connecting the dots belonging to the same pair. An example of a $(6,8)$-Brauer diagram is given below.
$$
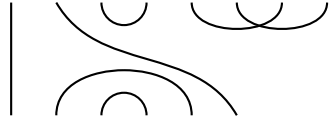
\begin{tikzpicture}[scale=1,thick,>=angle 90]
\begin{scope}[xshift=4cm]
\draw  (2.8,-0.5) -- +(0,1.5);
\draw (4,-0.5) to [out=90, in=180] +(0.3,0.3);
\draw (4.6,-0.5) to [out=90, in=0] +(-0.3,0.3);

\draw (3.4,-0.5) to [out=90, in=180] +(0.9,0.6);
\draw (5.2,-0.5) to [out=90, in=0] +(-0.9,0.6);

\draw (4,1) to [out=-90, in=180] +(0.3,-0.3);
\draw (4.6,1) to [out=-90, in=0] +(-0.3,-0.3);

\draw (5.2,1) to [out=-90, in=-90] +(1.2,0);

\draw (5.8,-0.5) to [out=120, in=-60] +(-2.4,1.5);

\draw (5.8,1) to [out=-90, in=-90] +(1.2,0);

\end{scope}
\end{tikzpicture}
$$ The lines in Brauer diagrams which connect the lower and upper horizontal line will be  referred to as {\em propagating lines}. 

The composition~$d_1\circ d_2$ of an $(i,j)$-diagram $d_1$ and a $(k,l)$-diagram $d_2$ is zero unless~$i=l$. When~$i=l$ we identify the dots on the upper line of~$d_2$ with those on the lower line of~$d_1$, creating another diagram. If this diagram contains loops, we have $d_1\circ d_2=0$. If it does not contain loops we obtain a $(k,j)$-Brauer diagram. Then~$d_1\circ d_2$ is equal to that diagram {\em up to a possible minus sign}. The rules for computing this minus sign were obtained in~\cite{Kujawa}. Note that {\it op. cit.} works with marked Brauer diagrams, whereas we follow the slightly different point of view that the homomorphisms are ordinary diagrams and their composition is to be determined by introducing the marking, see~\cite{PB1}.
The identity morphism of~$i\in\Ob\cA$ is the diagram with~$i$ non-crossing propagating lines, which we will denote by~$e_i^\ast$.

\subsubsection{Strict monoidal supercategory} \label{anti}
It is proved in~\cite[Theorem~3.2.1]{Kujawa}, see also \cite[Example~1.5(iii)]{supercategory}, that~$\cA$ is a strict monoidal supercategory. 
The superfunctor 
$$-\otimes- : \; \cA\,\boxtimes\,\cA\;\to\;\cA$$
satisfies~$i\otimes j =i+j$ for any $i,j\in\mN=\Ob\cA$. In particular, $\unit=0\in\Ob\cA$. Now we define the action of~$-\otimes-$ on morphisms. For any Brauer diagram $d$, we have that~$d\otimes e_i^\ast$, resp. $e_i^\ast \otimes d$, is the Brauer diagram obtained by adding $i$ propagating lines to the right, resp. the left, of~$d$. Now take an $(i,j)$-Brauer diagram $d_1$ and a $(k,l)$-Brauer diagram $d_2$. Then we set
$$d_1\otimes d_2\;=\; (d_1\otimes e_l^\ast)\circ (e_i^\ast\otimes d_2).$$ Thus $d_1\otimes d_2$ is again a diagram, up to a possible minus sign. The monoidal supercategory~$\cA$ is symmetric, with braiding morphisms $B_{i,j}: i\otimes j\to j\otimes i$ given in~\cite[Section~3.1]{Kujawa}.

By~\cite[Theorem~3.2.1]{Kujawa}, the monoidal supercategory~$\cA$ is generated by four morphisms:
\begin{enumerate}
\item $\II=e_1^\ast$, the identity morphism of~$1\in\Ob \cA$, represented by a straight line;
\item $\XX$, the crossing in~$\End_{\cA}(2)$; 
\item $\cup$, the unique diagram in~$\Hom_{\cA}(0,2)$; and
\item $\cap$, the unique diagram in~$\Hom_{\cA}(2,0)$.
\end{enumerate}

 \subsubsection{The periplectic Brauer algebra}\label{defAr}
The algebras in~\cite{Moon} are obtained as the endomorphism algebras in~$\cA$. We define the {\em periplectic Brauer algebra} as
$$A_r:=\End_{\cA}(r),\qquad\mbox{for $r\in \mN$}.$$
Note that the~$A_r$ are ordinary algebras with trivial $\mF_2$-grading, since all elements are even as noted in \ref{SSBrdiag}. The algebra~$A_r$ is for instance generated by the
elements 
$$s_i:=\II^{\otimes i-1}\otimes\XX\otimes\II^{\otimes  r-i-1}\quad\mbox{and}\quad \;\epsilon_i:= \II^{\otimes i-1}\otimes (\cup\circ \cap)\otimes\II^{\otimes  r-i-1},\quad\mbox{for~$1\le i<r$.}$$
 The subalgebra generated by~$\{s_i\}$ is precisely the symmetric group algebra~$\mk\mS_r$. The other relations are given in~\cite[Section~2]{Moon}.

From the monoidal structure on~$\cA$ we get an embedding of algebras
\begin{equation}\label{embedrs}A_r\otimes A_s\;\hookrightarrow\; A_{r+s}.\end{equation}
The embedding of~$A_r\otimes A_1=A_r\otimes\mk\II$ in~$A_{r+1}$ will simply be denoted by~$A_r\hookrightarrow A_{r+1}$. 

By \cite[Theorem~4.3.1]{Kujawa} or \cite[Theorem~1]{PB1}, the isoclasses of simple modules over $A_r$, with~$r\in\mN$, are in one-to-one correspondence with the following subset of~$\PA$:
\begin{equation}\label{eqLambdas}\Lambda_r:=\{\lambda\vdash j\;|\;j\in\JJJ^0(r) \}.\end{equation}
We denote the projective cover in~$A_r$-mod of the simple module~$L_r(\lambda)$, with~$\lambda\in\Lambda_r$, by~$P_r(\lambda)$. When~$\lambda\in\PA\backslash \Lambda_r$, we set~$L_r(\lambda)=P_r(\lambda)=0$.

\subsubsection{Cell modules}\label{SecCell} For~$r\in\mN$, we set
$\LL_r:=\{\lambda\vdash j\;|\;j\in\JJJ(r)\}.$
For any $\mu\in\LL_r$, the cell module~$W_r(\mu)$ was introduced in~\cite[Section~4]{PB1}. When~$\mu\in\PA\backslash \LL_r$, we set~$W_r(\mu)=0$.
 We use these modules to introduce a matrix $\cc$.
For~$\lambda,\mu\in\PA$, take an arbitrary~$r\in\mN$ with~$\lambda\in\Lambda_r$ and set
$$\cc_{\lambda\mu}\;:=\; [W_r(\mu):L_r(\lambda)].$$
The result in~\cite[Theorem~1]{PB2} shows in particular that the definition of~$\cc$ does not depend on the specific choice of~$r$. Furthermore, we have 
$$\cc_{\lambda\mu}:=\begin{cases}1&\mbox{if $\mu\subseteq \lambda$ and~$\lambda/\mu\in\Gamma$,}\\0&\mbox{otherwise.}\end{cases}$$
Here $\Gamma$ is a set of skew Young diagrams introduced in~\cite[Section~3]{PB2}. In particular, we have $\cc_{\lambda\lambda}=1$.
Since~$\cc_{\lambda\lambda}=1$ and~$\cc_{\lambda\mu}=0$ unless~$\mu\subseteq\lambda$, it is possible to construct a matrix $\cc^{\minus 1}$, such that
$\cc^{\minus 1}_{\lambda\lambda}=1$, $\cc^{\minus 1}_{\lambda\mu}=0$ unless~$\mu\subseteq\lambda$, and
$$\sum_{\mu}\cc_{\lambda\mu}\cc^{\minus 1}_{\mu\nu}\;=\;\delta_{\lambda\nu}\qquad\mbox{and}\qquad \sum_{\mu}\cc^{\minus 1}_{\lambda\mu}\cc_{\mu\nu}\;=\;\delta_{\lambda\nu},\qquad\mbox{for all $\lambda,\nu\in\PA$}.$$
Note that both summations are actually finite, by the lower diagonal structures of the matrices.

\subsubsection{Primitive idempotents and projective modules} \label{FixIdem}
Take an arbitrary partition~$\lambda$. We {\em fix for the remainder of the paper a primitive idempotent~$e_\lambda$ in~$A_j$ with~$j:=|\lambda|$}, according to the labelling in equation~\eqref{eqLambdas}. Hence we have $P_j(\lambda)\cong A_je_\lambda$. Examples of the idempotents are $e_{\varnothing}$, which is the identity element in~$\End_{\cA}(0)$, and~$e_{\Box}=\II$.

In \cite[Section~3]{PB1}, the algebra
$$C_r\;:=\;\bigoplus_{i,j\in\JJJ(r)}\Hom_{\cA}(i,j),$$
was introduced. By construction, we have $A_j\cong e_j^\ast C_r e_j^\ast$ for any $j\in\JJJ(r)$, which allows to interpret~$e_\lambda$ as an idempotent in~$C_r$ if $|\lambda|=j\in\JJJ(r)$.
By~\cite[Lemma~4.6.2]{PB1}, we have
\begin{equation}\label{ProjC}P_r(\lambda)\;\cong\; e_r^\ast C_r e_\lambda\;\cong\;\Hom_{\cA}(j,r) e_\lambda,\qquad\mbox{for all $\lambda\vdash j\in\JJJ^0(r)$}.\end{equation}

\subsubsection{Restriction and induction}
The embedding $A_r\hookrightarrow A_{r+1}$ of \ref{defAr} yields functors
$$\res_r: \; A_r\mbox{-mod}\to A_{r\,\minus\,1}\mbox{-mod}\qquad\mbox{and}\qquad\ind_r=A_{r+1}\otimes_{A_r}-: \; A_r\mbox{-mod}\to A_{r+1}\mbox{-mod}.$$
We introduce the symmetric matrix $\bb$ as
$$\bb_{\lambda\mu}\;=\; \begin{cases} 1&\mbox{if $\mu=\lambda\boxplus i$ or~$\mu=\lambda\boxminus i$, for some $i\in\mZ$,}\\0&\mbox{otherwise.}\end{cases}$$
By \cite[Corollary~5.2.4]{PB1}, $\res_r W_r(\mu)$ (for all $\mu\in\PA$ and~$r\in\mN$ such that~$r-|\mu|\in2\mZ_{>0}$) has a filtration with composition factors given by cell modules of~$A_{r\,\minus\,1}$ and multiplicities
\begin{equation}\label{ResEq}(\res_r W_r(\mu):W_{r\,\minus\,1}(\lambda))=\bb_{\lambda\mu},\quad\mbox{for all $\lambda\in\PA$.}\end{equation}
Note that multiplicities in cell filtrations of arbitrary~$A_r$-modules are actually independent of the chosen filtration, if $r\not\in\{2,4\}$, by \cite[Theorem~4.1.2(3)]{PB1}. 

\subsubsection{Jucys-Murphy elements}\label{JMel}
The Jucys-Murphy elements for~$A_r$ were introduced in~\cite[Section~6]{PB1}. The element $x_r\in A_r$ commutes with the subalgebra~$A_{r\,\minus\,1}$, by \cite[Lemma~6.1.2]{PB1}. We interpret~$x_r$ also as an element of~$A_s$ for any $r\ge s$, although $x_r\otimes e^\ast_{s-r}$ would be more precise. By definition, we have $x_1=0$.

We thus have an action of~$x_r$ on~$\res_rM$, for an $A_r$-module~$M$, which commutes with the~$A_{r\,\minus\,1}$-action.
In \cite[Section~2]{PB2}, we introduced the notation~$M_q$ for the generalised $q$-eigenspace for~$x_r$. We have $\res_rM=\oplus_{q\in\mZ}M_q$ as $A_{r\,\minus\,1}$-modules. 
For any $q\in\mZ$ and~$\lambda,\mu\in\PA$, we set
$$\bb^q_{\lambda\mu}=\begin{cases}1&\mbox{if $\mu=\lambda\boxplus  q$, or $\mu=\lambda\boxminus (q-1)$,}\\
0&\mbox{otherwise.}\end{cases}$$
Clearly, we have $\bb=\sum_{q\in\mZ}\bb^q$. By \cite[Proposition~2.12]{PB2}, we can refine~\eqref{ResEq} to
\begin{equation}\label{eqbq}(W_r(\mu)_{q}:W_{r\,\minus\,1}(\lambda))=\bb^q_{\lambda\mu}.\end{equation}

\subsection{The periplectic Lie superalgebra}\label{Secpn}

For each $n\in\mZ_{>0}$, the periplectic Lie superalgebra~$\pe(n)$ is the subalgebra of the general linear superalgebra~$\mathfrak{gl}(n|n)$, which preserves an odd bilinear form $\beta:V\times V\to\mk$, see~\cite{gang, Chen, PB1, Kujawa, Moon, Musson}, with~$V:= \mk^{n|n}$. Concretely,
$$\mathfrak{pe}(n)\;=\;\{X\in\mathfrak{gl}(n|n)\;|\;\beta(Xv,w)+(-1)^{|X||v|}\beta(v,Xw)=0,\quad\mbox{for all }\;v,w\in V\}.$$

\subsubsection{The supercategory $s\hspace{-0.3mm}\cF_n$ of integrable modules over $\pe(n)$}

We consider the category~$s\cF_n$ which has as objects all $\mF_2$-graded, finite dimensional, integrable, left $\pe(n)$-modules, see~\cite[Section~2]{gang}.
The morphism spaces consist of all $\mathfrak{pe}(n)$-linear morphisms of (ungraded) $\mk$-vector spaces.
 The morphism spaces are thus naturally $\mF_2$-graded vectorspaces. The category~$\Fn$ is a supercategory. 
  By `$\mathfrak{pe}(n)$-module' we will henceforth mean `object in~$\Fn$'.

 Note that there is a central element $H\in\mathfrak{pe}(n)_{\oa}\cong\mathfrak{gl}(n)$, whose adjoint action is diagonisable on~$\mathfrak{pe}(n)_{\ob}$ with eigenvalues $\pm 1$. This allows to equip any weight module $M$ with a $\mF_2$-grading. For instance, we can set $M_{\oa}$, resp. $M_{\ob}$, equal to the sum of all weight spaces for weights $\lambda$ such that~$\lambda(H)$ is even, resp. odd.
 It then follows easily that~$\Fn$ is abelian.
  
 In order to be compatible with \cite{Kujawa}, we will think of morphisms as `acting from the right' and denote by 
 $$\Hom_{\Fn}(M,N)\;=\;\Hom_{\pe(n)}(N,M),$$ the space of $\pe(n)$-linear morphisms $N\to M$.
Hence we write $(v)f$, for $v\in N$ and $f\in \Hom_{\Fn}(M,N)$. For $f\in \Hom_{\Fn}(M,N)$ and $g\in \Hom_{\Fn}(K,M)$, the morphism $f\circ g\in \Hom_{\Fn}(K,N)$ is given by
$$(v)(f\circ g)=((v)f)g,\qquad\mbox{for $v\in N$.}$$

 \subsubsection{Monoidal structure}

For~$\mathfrak{pe}(n)$-modules $M,N$, the tensor product $M\otimes N=M\otimes_{\mk} N$ is an object in~$\Fn$, with action given by
$$X(v\otimes w)\;=\;Xv\otimes w\;+\;(-1)^{|X||v|}v\otimes Xw,\qquad\mbox{for all }\;X\in\pe(n).$$
For~$f\in \Hom_{\pe(n)}(M_1,M_2)$ and~$g\in \Hom_{\pe(n)}(N_1,N_2)$, the morphism $f\otimes g$ defined as
$$(v\otimes w)(f\otimes g)\;=\; (-1)^{|f||w|} (v)f\otimes (w)g$$
is $\mathfrak{pe}(n)$-linear.
With this rule, equation~\eqref{superComp} is satisfied and
and $\Fn$ is a monoidal supercategory for~$-\otimes-$.


\section{The periplectic Deligne category}\label{SecDelCat}

\subsection{Construction}
The category~$\DA$, which we will define as the pseudo-abelian envelope of~$\cA$, was denoted by~${\rm \underline{Rep}}\,P$ in~\cite[Section~4.5]{Vera} and by~$\widehat{\cB}(0,\minus1)$ in~\cite[Section~5]{Kujawa}. It is the periplectic analogue of the categories~${\rm \underline{Rep}}\,GL_\delta$ and~${\rm \underline{Rep}}\,O_\delta$ introduced by Deligne in~\cite{Deligne}.

\subsubsection{}The periplectic Brauer category~$\cA$ is $\mk$-linear, so in particular pre-additive. It thus admits a unique (up to equivalence) additive envelope. We define such a supercategory~$\overline{\cA}$ which has as objects finite multisets of elements in~$\mN=\Ob\cA$.
For such a multiset $S$, the corresponding objet of $\overline{\cA}$ is denoted by~$\bigoplus_{r\in S}r$. Morphisms in~$\overline{\cA}$ are matrices with entries morphisms in~$\cA$. By construction, $\overline{\cA}$ is still skeletal. It is an additive category, with biproducts given by
$$\left(\bigoplus_{r\in S}r\right)\oplus \left(\bigoplus_{r\in S'}r\right)\;=\;\bigoplus_{r\in S\sqcup S'}r.$$ The category~$\overline{\cA}$ inherits a structure of a symmetric strict monoidal supercategory from its subcategory~$\cA$, with~$-\otimes-$ extended as a bi-additive functor. 

\subsubsection{} The additive category~$\overline{\cA}$ admits a unique (up to equivalence) Karoubi envelope. We define~$\DA$ with objects all pairs $(X,e)$ with~$X\in\Ob\overline{\cA}$ and~$e$ an idempotent in~$\End_{\overline{\cA}}(X)$. Morphism superspaces in~$\DA$ are given by
\begin{equation}\label{MorDeli}\Hom_{\DA}((X,e),(Y,f))\;=\; \{\alpha\in \Hom_{\overline{\cA}}(X,Y)\,|\, \alpha=f\circ \alpha\circ e\}\;=\; f\Hom_{\overline{\cA}}(X,Y)e.\end{equation}
Since~$\DA$ is karoubian, additive and~$\mk$-linear with finite dimensional endomorphism algebras, it is Krull-Schmidt. It also inherits naturally from its subcategory~$\overline{\cA}$ the structure of a symmetric monoidal supercategory,  with~$-\otimes-$ a bi-additive functor. For~$i,j\in \mN=\Ob\cA\subset\Ob\overline{\cA}$ and idempotents~$e\in A_i$ and~$f\in A_j$, we have for instance
\begin{equation}\label{eqtens}(j,f)\otimes (i,e)\;=\;(j+i, f\otimes e),\end{equation}
with~$f\otimes e$ interpreted as an element in~$A_{j+i}$ as in \eqref{embedrs}. 
\begin{rem} Consider the category $\cS:=\bigoplus_{i\in\mN}\mk\mS_i$-mod.
Since, $\charr(\mk)=0$, $\cS$ is a pseudo-abelian envelope of the~$\mk$-linear subcategory~$\cC$ of $\cA$ with objects~$\mN$, $\Hom_{\cC}(i,j)=0$ if $i\not=j$ and~$\End_{\cC}(i)=\mk\mS_i$.
 It is in this spirit that our categorical representation on~$\DA$ is an analogue of the one for~$\mathfrak{sl}_{\infty}$ on~$\cS$ in~\cite{LLT, Oded}.
\end{rem}

\subsection{Indecomposable objects and blocks}

For any $\lambda\in\PA$, we set
$$R(\lambda):= (|\lambda|, e_\lambda)\;\in\,\DA,$$
with~$e_\lambda$ the primitive idempotent in~$A_{|\lambda|}$ of \ref{FixIdem}. In particular, $R(\varnothing)=\unit$ and~$R(\Box)=(1,\II)$.

\begin{thm}\label{ThmClass}
The assignment $\lambda\mapsto R(\lambda)$ gives a bijection between $\PA$ and the set of isomorphism classes of
non-zero indecomposable objects in~$\DA$.
\end{thm}
\begin{proof}
Let~$X$ be an arbitrary non-zero indecomposable object in~$\DA$. Clearly $X$ is isomorphic to~$(r,e)$ for some~$r\in \mN=\Ob\cA\subset\Ob\overline{\cA}$ and~$e$ a primitive idempotent in~$A_r=\End_{\overline\cA}(r)$.  If $r=0$, then~$X=R(\varnothing)$, so we can assume $r>0$.
Let~$\mu$ be the partition of~$r-2i\in\JJJ^0(r)$ such that~$A_re\cong P_r(\mu)$. We will show in two steps that~$R(\mu)\cong X$.

By \cite[4.2.1]{PB1}, there exist $a\in\Hom_{\cA}(r-2i,r)$ and~$b\in\Hom_{\cA}(r,r-2i)$ such that~$ba=e^\ast_{r-2i}$. Consequently, $\overline{e}_\mu:= a e_\mu b $ is an idempotent in~$A_r$. We define, using~\eqref{MorDeli},
\begin{eqnarray*}
x:=e_\mu b=e_{\mu}b\overline{e}_{\mu}&\in& e_\mu\Hom_{\cA}(r,r-2i)\overline{e}_\mu =\Hom_{\DA}((r, \overline{e}_\mu),R(\mu))\quad\mbox{and}\\
 y:=ae_\mu=\overline{e}_\mu a e_\mu&\in& \overline{e}_\mu\Hom_{\cA}(r-2i,r)e_\mu= \Hom_{\DA}(R(\mu),(r,\overline{e}_\mu)).
 \end{eqnarray*}
Since~$xy=e_\mu$ and~$yx=\overline{e}_\mu$, the identity morphisms of~$R(\mu)$ and~$(r,\overline{e}_\mu)$, we have~$R(\mu)\cong (r,\overline{e}_\mu)$.

By equation~\eqref{ProjC} and the properties of~$a$ and~$b$, we have isomorphisms of left $A_r$-modules:
$$A_re\;\cong\;e_r^\ast C_re_\mu\;\cong\; A_r \overline{e}_\mu$$
This means that there exist $\alpha\in  e A_r\overline{e}_\mu$ and~$\beta\in \overline{e}_\mu A_re$, corresponding to the mutual inverses in 
$$e\Hom_{\cA}(r,r)\overline{e}_\mu \;=\;\Hom_{\DA}((r,\overline{e}_\mu), (r,e))\quad\mbox{and}\quad \overline{e}_\mu\Hom_{\cA}(r,r)e \;=\;\Hom_{\DA}((r,e), (r,\overline{e}_\mu)).$$
Hence~$(r,e)\cong(r,\overline{e}_\mu)\cong R(\mu)$, so we find that any indecomposable object in~$\DA$ is isomorphic to some~$R(\lambda)$.

Now assume that for~$\lambda\not=\mu$ we have $R(\mu)\cong R(\lambda)$. The corresponding isomorphism which must exist in~$e_\mu\Hom_{\cA}(t,s)e_\lambda$ with~$t=|\lambda|$ and~$s=|\mu|$ implies that~$t-s$ is even and that~$C_re_\lambda\cong C_re_\mu$ in~$C_r$-mod, for~$r$ such that~$s,t\in\JJJ(r)$. This is contradicted by \cite[Section~3]{PB1}.\end{proof}

\begin{rem}\label{RemThm}
The proof of Theorem~\ref{ThmClass} implies that for an arbitrary primitive idempotent $e\in A_r$, we have $R(\lambda)\;\cong (r,e)$ if and only if $A_re\cong P_r(\lambda)$.
\end{rem}

\begin{cor}
When neglecting the monoidal structure, we have an equivalence of categories 
$$\DA\;\cong\;\bigoplus_{k\in\mN}\DA[k],$$
with~$\DA[k]$ the full additive, indecomposable, subcategory of $\DA$ containing all $R(\lambda)$ where the 2-core of $\lambda$ is~$\partial^k$.
\end{cor}
\begin{proof}
Since $\cA$ decomposes into the coproduct of two subcategories, corresponding to the even and odd integers, we know that $\DA$ decomposes similarly. Now take two partitions $\lambda,\mu$ with~$|\lambda|-|\mu|$ even.
By Remark~\ref{RemThm}, there exists $r\in\mN$ and idempotents $e,f\in A_r$ for which
$$\Hom_{\DA}(R(\lambda),R(\mu))\;\cong\;\Hom_{\DA}((r,e),(r,f))\;\cong\; fA_re\;\cong\; \Hom_{A_r}(P_r(\lambda),P_r(\mu)).$$
The block decomposition of $\DA$ is thus inherited from the one of $A_r$ in \cite[Theorem~1]{PB1}.
\end{proof}

\subsection{The split Grothendieck group}
We let~$[\DA]_{\oplus}$ denote the split Grothendieck group of the small additive category~$\DA$, see \cite[Section~1.2]{Maz}. Concretely,  $[\DA]_{\oplus}$ is the free abelian group with elements the isomorphism classes $[X] $ of objects~$X$ in~$\DA$, modulo the relations~$[X]=[ Y]+[ Z]$, whenever $X=Y\oplus Z$.
As an immediate consequence of Theorem~\ref{ThmClass}, we thus find the following.
\begin{cor}\label{CorSplit}
The map $\Psi:\PA_{\mZ}\to [\DA]_{\oplus}$ determined by~$\lambda\mapsto [R(\lambda)]$ is a $\mZ$-module isomorphism.
\end{cor}
In the terminology of \cite[Section~1.3]{Maz}, $(\DA,\Psi)$ is a $\mZ$-categorification of~$\PA_{\mZ}$.


\section{Tensor product with the generator}\label{SecTensT}
In this section, we study the functor~$\TT$, the endo-superfunctor of~$\DA$ given by
$$\TT(-):= -\otimes R(\Box).$$
For idempotents~$e\in A_r$, $f\in A_s$ and~$a\in f\Hom_{\cA}(r,s)e=\Hom_{\DA}((r,e),(s,f))$, we thus have
\begin{equation}\label{eqDefT}\TT(r,e)\;=\;(r+1,e\otimes \II)\qquad\mbox{and}\qquad \TT(a)=a\otimes\II, \end{equation}
by definition and equation~\eqref{eqtens}.

\subsection{The combinatorics of~$\TT$}
We use the Krull-Schmidt category~$\DA$ to define a matrix $\aac$.
\begin{ddef}\label{defa}
For all $\nu,\mu\in \PA$, we define $\aac_{\nu,\mu}\in\mN$, by
$$\TT(R(\nu))\;=\;R(\nu)\otimes R(\Box)\;\cong\;\bigoplus_{\kappa} R(\kappa)^{\oplus \aac_{\nu\kappa}}.$$
\end{ddef}
Recall the matrices $\bb$ and~$\cc$ introduced in Section~\ref{ACat}.
\begin{thm}\label{ThmTensorB}We have $\aac=\cc\,\bb\,\cc^{\minus1}$. Concretely,
for all~$\nu,\kappa\in \PA$, we have
\begin{equation}\label{eqCmat}\aac_{\nu\kappa}\;=\;\sum_{\lambda\subseteq \nu}\sum_{\mu\supseteq\kappa}\,\cc_{\nu\lambda}\,\bb_{\lambda\mu} \cc^{\minus 1}_{\mu\kappa}.\end{equation}\end{thm}
\begin{proof}
Take $r=|\nu|$.
Equation~\eqref{eqDefT} and Remark~\ref{RemThm} imply that~$\aac_{\nu\kappa}$ is the number of times the projective $A_{r+1}$-module~$P_{r+1}(\kappa)$ appears as a direct summand of 
$$A_{r+1}(e_\nu\otimes I)\;\cong\; \ind_{r} P_r(\nu).$$
Consequently
\begin{equation}\label{aSimple}\aac_{\nu\kappa}\;=\;\dim\Hom_{A_{r+1}}(\ind_{r} P_r(\nu),L_{r+1}(\kappa))\;=\;[\res_{r+1}L_{r+1}(\kappa):L_{r}(\nu)].\end{equation}
In particular, we thus find
$$\sum_{\kappa}\aac_{\nu\kappa}\cc_{\kappa\lambda}\;=\;\sum_\kappa[W_{r+1}(\lambda):L_{r+1}(\kappa)][\res_{r+1}L_{r+1}(\kappa):L_r(\nu)]\;=\;[\res_{r+1}W_{r+1}(\lambda):L_r(\nu)].$$
On the other hand, by equation~\eqref{ResEq}, we have
$$\sum_{\mu}\cc_{\nu\mu}\bb_{\mu\lambda}\;=\;\sum_{\mu}[\res_{r+1}W_{r+1}(\lambda):W_r(\mu)][W_{r}(\mu):L_r(\nu)]\;=\;[\res_{r+1}W_{r+1}(\lambda):L_r(\nu)].$$
This shows that~$\aac=\cc\,\bb\,\cc^{\minus1}$.
\end{proof}

\begin{rem}
Equation~\eqref{aSimple} shows the explicit connection between $\TT$ on~$\DA$ and~$\res$ between the Brauer algebras. This explains the similarities between translation functors for the periplectic Lie superalgebra \cite[Corollary~4.4.6]{gang} and  the restriction functors \cite[Proposition~2.3.1]{PB2}. \end{rem}

\subsection{The natural transformation~$\xi:\TT\Rightarrow\TT$}\label{SecXi}

For an object $X=(r,e)$ in~$\DA$, we define 
$$\xi_X\in\End_{\DA}(X\otimes R(\Box))=(e\otimes \II)A_{r+1}(e\otimes \II),\qquad \mbox{as}$$
$$ \xi_X\;=\;(e\otimes \II)x_{r+1}(e\otimes \II)\;=\;(e\otimes \II)x_{r+1}\;=\; x_{r+1}(e\otimes \II),$$
with~$x_{r+1}\in A_{r+1}$ the Jucys-Murphy element. The different identities for~$\xi_X$ are equal since~$x_{r+1}$ commutes with elements of~$A_r$.
We can easily extend this to arbitrary objects~$X$ in~$\DA$.
\begin{prop}
The family of morphisms $\{\xi_X\,|\, X\in\Ob\DA\}$ yields an even natural transformation of the superfunctor~$\TT$ on~$\DA$.
\end{prop}
\begin{proof}
Consider objects~$X=(r,e)$, $Y=(s,f)$ and a morphism 
$$a\;\in\;\Hom_{\DA}(X,Y)=f\Hom_{\cA}(r,s)e.$$
We claim that~$\TT(a)\circ \xi_X=\xi_Y \circ\TT(a)$.
Indeed, by \eqref{eqDefT} the left-hand, resp. right-hand side, becomes
$$(a\otimes \II)(e\otimes \II)x_{r+1}\;=\;(a\otimes \II)x_{r+1},\quad\mbox{resp.}\quad x_{s+1}(f\otimes\II)(a\otimes \II)\;=\;x_{s+1}(a\otimes \II).$$
The claim then follows from the subsequent Lemma~\ref{LemJM}. 
\end{proof}

\begin{lemma}
\label{LemJM}
For arbitrary~$a\in \Hom_{\cA}(s,r)$, we have
$(a\otimes \II)x_{r+1}=x_{s+1}(a\otimes \II).$
\end{lemma}
\begin{proof}
The case~$r=s$ is precisely the fact that~$x_{r+1}$ commutes with~$A_r$, see \ref{JMel}. This means that it suffices to prove that, for $r\ge 2$,
$$(\cup\otimes \II^{\otimes r\,\minus\,1})x_{r\,\minus\,1}\;=\; x_{r+1}(\cup\otimes \II^{\otimes r\,\minus\,1})\quad\mbox{and}\quad(\cap\otimes \II^{\otimes r\,\minus\,1})x_{r+1}\;=\; x_{r\,\minus\,1}(\cap\otimes \II^{\otimes r\,\minus\,1}).$$
These easy calculations are left to the reader.
\end{proof}

\subsection{The functors~$\TT_q$}
We introduce some elements of~$A_r$. On any $A_r$-module, $x_r\in A_r$ only attains integer eigenvalues, see \cite[Section~6.2]{PB1}. If $r>0$, we can thus construct mutually orthogonal idempotents~$\gamma^{(r)}_i\in A_r$, for~$i\in\mZ$, which are in the subalgebra generated by~$x_r$, such that
\begin{equation}\label{eqsum1}1_{A_r}\;=\;e_r^\ast\;=\;\sum_{i\in\mZ}\gamma^{(r)}_i,\qquad\mbox{and~$\;\;\;\;\;\;\,(x_r-i)^p\gamma^{(r)}_i=0$, for some $p\in\mN$. }\end{equation}
Since we keep track of~$r$ in the notation, we can with slight abuse of notation also write $\gamma_j^{(r)}$ for~$\gamma_j^{(r)}\otimes e^\ast_{s-r}\in A_s$. By construction, $\gamma^{(r)}_i$ commutes with any element of~$A_{r\,\minus\,1}\subset A_r$. We also set~$\gamma_i^{(0)}=\delta_{i0}\in \mk =A_1$.
\begin{ex}
We have $x^2_2=1$ and consequently $\gamma^{(2)}_{1}=\frac{1}{2}(1+x_2)$, $\gamma^{(2)}_{-1}=\frac{1}{2}(1-x_2)$ and~$\gamma^{(2)}_i=0$ if $i\not\in\{1,-1\}$.
\end{ex}
For an idempotent $e\in A_r$, we set
$$e[j]\;=\; \gamma_j^{(r+1)}(e\otimes\II)\;=\;(e\otimes\II)\gamma_j^{(r+1)}.$$

\begin{ddef}\label{defaq}
For any $j\in\mZ$, the additive functor~$\TT_j$ on~$\DA$ is defined as
$$\TT_j(r,e)\;=\; (r+1,e[j]),\qquad\mbox{for all $r\in\mN$ and~$e$ an idempotent in~$A_r$, and}$$
$$\TT_j(a)\;=\; f[j](a\otimes \II)e[j]=\gamma_j^{(s+1)}(a\otimes\II)\gamma_j^{(r+1)},\qquad\mbox{for all $a\in\Hom_{\DA}((r,e),(s,f))$.}$$
By construction, we have $\TT=\bigoplus_{j\in\mZ}\TT_j$. Following Definition~\ref{defa}, for each $q\in\mZ$, we define a matrix $\aac^q$ by
$$\TT_q(R(\nu))\;=\; \bigoplus_{\kappa}R(\kappa)^{\oplus \aac^q_{\nu\kappa}}.$$
\end{ddef}

\begin{prop}\label{Propaq}
For each $q\in\mZ$, we have $\aac^q=\cc\,\bb^q\,\cc^{\minus1}$.
\end{prop}
\begin{proof}
This is an analogue of the proof of Theorem~\ref{ThmTensorB}. Consider $\nu\in\PA$ with~$r=|\nu|$. We have
$$a_{\nu\kappa}\,=\, \dim_{\mk} (e_\nu\otimes\II)L_{r+1}(\kappa)\,=\, \dim_{\mk} e_\nu\,\res_{r+1}L_{r+1}(\kappa).$$
Correspondingly, we find
$$a^q_{\nu\kappa}\,=\, \dim_{\mk} e_\nu[q] L_{r+1}(\kappa).$$
Since~$e_\nu[q](x_{r+1}-q)^k=0$ for some $k\in\mN$, we find
\begin{equation}\label{eqaq}a^q_{\nu\kappa}\,=\, \dim_{\mk} e_\nu L_{r+1}(\kappa)_{q}\,=\,\dim_{\mk}\Hom_{A_r}(A_re_\nu,L_{r+1}(\kappa)_{q})\,=\,[L_{r+1}(\kappa)_{q}:L_r(\nu)].\end{equation}
This and equation~\eqref{eqbq} imply that
$$(\aac^q\,\cc)_{\nu\lambda}\;=\;[W_{r+1}(\lambda)_{q}:L_r(\nu)]\;=\; (\cc\,\bb^q)_{\nu\lambda},$$
which concludes the proof.
\end{proof}

  \begin{lemma}\label{RemoveBox}
 Let~$\nu$ be a partition with~$\kappa=\nu\boxminus q$. 
 \begin{enumerate}[(i)]
 \item If $\kappa$ has a removable~$q-1$-box, then~$\aac^{q\,\minus \,1}_{\nu\kappa}=1$.
  \item  If $\kappa$ has a removable~$q+1$-box, then~$\aac^{q+1}_{\nu\kappa}=1$.
 \end{enumerate}
  \end{lemma}
 \begin{proof}
 Part (i) is  \cite[Lemma~2.2.3]{PB2}, by equation~\eqref{eqaq}.

 For part (ii), we claim that~$\nu$ does not admit an addable~$q+1$-box. Indeed, in order for~$\kappa$ to have a removable~$q+1$-box, there must be a $q+1$-box above the~$q$-box in~$\nu$, such that there is no $q+2$-box to the right of the~$q+1$-box. 
Part (ii) then follows from 
 \cite[Lemma~2.2.1]{PB2}, by equation~\eqref{eqaq}.
 \end{proof}

 An alternative way to prove Lemma~\ref{RemoveBox} is to use the results in Section~\ref{SecTL}.
 
 \begin{lemma}\label{Lemaa1}
 If $\tilde\lambda=\lambda\boxplus q$, then~$\aac^q_{\lambda\tilde{\lambda}}=1$.
 \end{lemma}
 \begin{proof}
 By Proposition~\ref{Propaq}, we have
 $$\aac^q_{\lambda\tilde{\lambda}}=\sum_{\mu\subseteq \lambda}\sum_{\nu\supseteq \tilde{\lambda}}\cc_{\lambda\mu}\bb^q_{\mu\nu}\cc^{\minus 1}_{\nu\tilde{\lambda}}.$$
 The summation thus goes over $\mu,\nu\in\PA$ with
 $\mu\subseteq \lambda\subsetneq \tilde{\lambda}\subseteq\nu.$
 On the other hand~$\bb^{q}_{\mu\nu}=0$ unless~$\mu$ and~$\nu$ differ by precisely one box. Hence we have
 $$\aac^q_{\lambda\tilde{\lambda}}=\cc_{\lambda\lambda}\bb^q_{\lambda\tilde{\lambda}}\cc^{\minus 1}_{\tilde{\lambda}\tilde{\lambda}}=1,$$
 which concludes the proof.
 \end{proof}
 
 \begin{cor}\label{Coraa1}
If $\tilde\lambda=\lambda\boxplus q$, then~$(\cc\,\bb^q)_{\lambda\nu}\ge \cc_{\tilde{\lambda}\nu}$ for all $\nu\in\PA$.
 \end{cor}
 \begin{proof}
 By Lemma~\ref{Lemaa1}, we have $\aac^q_{\lambda\eta}\ge\delta_{\eta\tilde{\lambda}}$, for all $\eta\in\PA$, where positivity of the entries of~$\aac^{q}$ follows by Definition~\ref{defaq}. Since the entries of~$\cc$ are also positive, we thus find
$$(\cc\bb^q)_{\lambda\nu}\;=\;(\aac^q\cc)_{\lambda\nu}\;\ge\; \cc_{\tilde{\lambda}\nu},$$
where the first equation is Proposition~\ref{Propaq}.
 \end{proof}
 

\section{The Fock space representation of the infinite Temperley-Lieb algebra}\label{SecTL}
Consider the~$\mZ$-algebra with generators~$\{T_i\,|\,i\in\mZ\}$ and relations (with~$|i-j|>1$)
$$T_i^2=0,\;\; \;\;T_iT_j=T_jT_i\;\;\;\mbox{and}\;\;\;\;\; T_{i}T_{i\pm 1}T_i=T_i.$$
This is the infinite Temperley-Lieb algebra over $\mZ$ for parameter zero, $\TL_\infty(0)$. In this section, we will consider two representations of~$\TL_\infty(0)$ on~$\PA_{\mZ}$, related by an automorphism of~$\PA_{\mZ}$. Due to its close connection with the Fock space representation of~$\mathfrak{sl}_\infty$, we will refer to one as the Fock space representation of~$\TL_\infty(0)$. The twisted version is the one that will describe the combinatorics of the periplectic Deligne category and will be referred to as $\Xi$.

\subsection{The representation~$\Xi$}
By Propositions~\ref{PropUnique}, \ref{PropExis} and~\ref{PropFaith}, we have the following theorem.
\begin{thm}
There exists a unique representation 
$$\Xi\,:\;\TL_{\infty}(0)\,\to\,\End_\mZ(\PA_{\mZ})$$
which satisfies for all $q\in\mZ$:
\begin{itemize}
\item $\Xi(T_q)(\varnothing)=\delta_{q0}\Box$;
\item $\Xi(T_q)(\lambda)\;=\;\lambda\boxplus q$ for any $\lambda\in\PA$ which admits an addable~$q$-box.
\end{itemize}
Moreover, the representation~$\Xi$ is faithful.
\end{thm}

The following theorem is an immediate consequence of the realisation of~$\Xi$ in Proposition~\ref{PropExis}.
\begin{thm}\label{ThmDecat}
The $\mZ$-module isomorphism $\Psi: \PA_{\mZ}\to[\DA]_{\oplus}$ in Corollary~\ref{CorSplit} satisfies~$[\TT_j]\circ\Psi=\Psi\circ\Xi(T_j)$.
Hence, the functors~$\TT_j$ satisfy the properties (with~$|i-j|>1$)
$$[\TT_i]^2=0,\;\; \;\;[\TT_i][\TT_j]=[\TT_j][\TT_i]\;\;\;\mbox{and}\;\;\;\;\; [\TT_{i}][ \TT_{i\pm 1}][\TT_i]=[\TT_i].$$
\end{thm}
This means that~$(\DA,\Psi,\{\TT_i\,|\,i\in\mZ\})$ is a $\mZ$-categorification of the $\TL_{\infty}(0)$-representation~$\Xi$, in the na\"ive sense in the terminology of~\cite[Section~2.2]{Maz}. We will improve this statement in Section~\ref{SecCateg}.

 In the following, we will usually write $T_q(\lambda)$ instead of~$\Xi(T_q)(\lambda)$.

\subsection{Uniqueness of the representation}
The combinatorial arguments in this subsection are inspired by the results and proofs in~\cite[Section~7.2]{gang}.

\subsubsection{} For aribrary~$\lambda\in\PA$ and $q\in\mZ$, there are 5 mutually exclusive possibilities:
\begin{enumerate}[(a)]
\item $\lambda$ admits an addable~$q$-box;
\item $\lambda$ has a removable~$q$-box;
\item $\lambda$ has a no boxes with content in~$\{q-1,q,q+1\}$ (and~$\lambda\not=\varnothing$ when~$q=0$);
\item there is a box right of the (existing) rim $q$-box of~$\lambda$, but not below;
\item there is a box below the (existing) rim $q$-box of~$\lambda$, but not to its right.
\end{enumerate}
We draw the~$q-1$, $q$ and~$q+1$ boxes on the rim of~$\lambda$ in the `generic' cases (meaning assuming that all three contents appear in~$\lambda$) corresponding to (a), (b), (d) and~(e):
\[
(a):\parbox[c]{1cm}{ \young(q\leeg,\leeg)},\quad (b): \parbox[c]{1cm}{\young(:\leeg,\leeg q)},\quad
(d):\parbox[c]{1.5cm}{\young(\leeg q\leeg) },\quad (e):\parbox[c]{0.7cm}{\young(\leeg,q,\leeg)}.
\]
If it is clear from context which $q$ is referred to, we will simply say that~$\lambda$ is of type (a), (b), etc.

\subsubsection{} We also introduce some terminology for (rim) hooks. A hook is called {\em balanced} if its height (the number of rows it has boxes in) is the same as its width (the number of columns it has boxes in). A rim hook of~$\lambda$ such that the minimal, resp. maximal, content of its boxes is $q$ is called {\em a rim hook starting at~$q$}, resp. {\em a rim hook ending at~$q$}.

In case~(d) there will always be a rim hook starting at~$q$ and one starting at~$q+1$, in case~(e) there will always be a rim hooks ending at~$q$ and~$q-1$.

\begin{prop}\label{PropUnique}
Assume that a representation of~$\TL_{\infty}(0)$ on~$\PA_{\mZ}$ satisfies, for any $q\in\mZ$ and~$\lambda\in\PA$:
\begin{enumerate}[(I)] 
\item $T_q(\varnothing)=\delta_{q0}\Box$;
\item $T_q(\lambda)\;=\;\lambda\boxplus q$ if $\lambda$ is of type (a).
\end{enumerate}
Then we have the following:
\begin{enumerate}
\item[(III)] If $\lambda$ is of type (b) or (c), $T_q(\lambda)\;=\;0$;
\item[(IV)] If $\lambda$ is of type (d), $T_q(\lambda)$ is the partition obtained by removing the minimal balanced rim hook starting at~$q+1$, if that exists, otherwise $T_q(\lambda)=0$;
\item[(V)] If $\lambda$ is of type (e), $T_q(\lambda)$ is the partition obtained by removing the minimal balanced rim hook ending at~$q-1$, if that exists, otherwise $T_q(\lambda)=0$.
\end{enumerate}
In particular, there is at most one representation of~$\TL_{\infty}(0)$ on~$\PA_{\mZ}$ satisfying (I) and~(II).
\end{prop}
We prove this in four lemmata and denote by~$\Omega$ an arbitrary representation of~$\TL_\infty(0)$ on~$\PA_\mZ$.

\begin{lemma}\label{Lem1}
Assume that~$\Omega$ satisfies (I) and~(II), then it satisfies (III).
\end{lemma}
\begin{proof}
Assume first that~$\lambda$ has a removable~$q$-box (type (b)). Then (II) implies that~$\lambda=T_q(\mu)$ for $\mu=\lambda\boxminus q$. Hence we find $T_q(\lambda)=T^2_q(\mu)=0$. 

Now assume that~$\lambda$ is of type (c). Then~$\lambda=T_{p_1}T_{p_2}\cdots T_{p_k}(\varnothing)$, with~$k=|\lambda|$ and each $p_i\not\in\{q-1,q,q+1\}$ by (II). The Temperley-Lieb relations thus imply that
$$T_q(\lambda)\;=\;T_{p_1}T_{p_2}\cdots T_{p_k}T_q(\varnothing). $$
As we can clearly assume that~$q\not=0$, this must be zero by (I), which concludes the proof.
\end{proof}

If $\lambda$ is of type (d) for~$q$, we let~$t\in\mN$ denote the maximal number such that there is a box in~$\lambda$ with content $q+t+1$ on the row of the rim $q$-box. We then specify that~$\lambda$ {\em is of type (d,[$r$,$t$])}, with~$r=|\lambda|$.

\begin{lemma}\label{Lem2}
Assume that~$\Omega$ satisfies (II), then it satisfies condition (IV) for all $\lambda$ of types (d,$[r,0]$) and~(d,$[0,t]$).
\end{lemma}
\begin{proof}
If $\lambda$ is of type (d,$[r,0]$), then~$\lambda$ has a removable~$q+1$-box, so by (II) we have $\lambda=T_{q+1}(\mu)$, with~$\mu=\lambda\boxminus (q+1)$. Furthermore, $\mu$ has a removable~$q$-box, so 
$$\lambda=T_{q+1}T_q(\nu)$$
with~$\nu$ obtained from~$\lambda$ by removing the~$q$ and~$q+1$ boxes on its rim. We thus find
$$T_q(\lambda)=T_qT_{q+1}T_q(\nu)=T_q(\nu)=\mu.$$
In conclusion, $T_q(\lambda)=\lambda\boxminus (q+1)$. Clearly, that removed rim $q+1$-box is the minimal balanced rim hook starting at~$q+1$.

The case~(d,$[0,t]$) is empty, since~$\lambda=\varnothing$ is never of type (d). 
\end{proof}

\begin{lemma}\label{Lem3}
Assume that~$\Omega$ satisfies (II), (III) in general and~(IV) for all partitions of type (d,$[r',-]$) with~$r'<r$, then it satisfies condition (IV) for~$\lambda$ of type (d,$[r,1]$).
\end{lemma}
\begin{proof}
By assumption on~$\lambda$ and~(II) we have $\lambda=T_{q+2}(\mu)$, with~$\mu=\lambda\boxminus (q+2)$. Hence we have
$$T_q(\lambda)=T_{q}T_{q+2}(\mu)=T_{q+2}T_{q}(\mu).$$
Furthermore $\mu$ is of type (d,$[r\,\minus\,1,0]$), so (IV) holds true which means $T_q(\mu)=\nu$, with~$\nu=\mu\boxminus (q+1)$. Hence, we have
$$T_q(\lambda)\;=\;T_{q+2}(\nu).$$
We review the two possibilities for~$\nu$. 

If the rim $q$-box of~$\lambda$ was on the highest row, then~$\nu$ contains no box with content in~$\{q+1,q+2,q+3\}$, so $T_q(\lambda)=T_{q+2}(\nu)=0$ by (III). In this case, $\lambda$ has no balanced rim hook starting at~$q+1$, so (IV) is indeed satisfied for~$\lambda$. 

If there is a row above the rim $q$-box, then~$\nu$ is clearly again of type (d), now for~$q+2$. Furthermore, $|\nu|<|\lambda|=r$ so $\nu$ satisfies (IV). Moreover, we have a clear one-to-one correspondence between the rim hooks of~$\nu$ starting at~$q+3$ and the rim hooks of~$\lambda$ starting at~$q+1$, by adding the rim $q+1$ and $q+2$-boxes in~$\lambda$ to the former hook. This correspondence preserves the notion of balancedness. Hence we find that~$\lambda$ satisfies (IV).
\end{proof}

\begin{lemma}\label{Lem4}
Assume that~$\Omega$ satisfies (II) in general and~(IV) for all partitions  of type (d,$[r',-]$) with~$r'<r$, then it satisfies condition (IV) for~$\lambda$ of type (d,$[r,t]$) with~$t>1$.
\end{lemma}
\begin{proof}
We have $\lambda=T_{q+t+1}(\mu)$, with~$\mu=\lambda\boxminus (q+t+1)$. We thus have
$$T_q(\lambda)=T_qT_{q+t+1}(\mu)=T_{q+t+1}T_q(\mu),$$
where now $\mu$ is of type (d,$[r\,\minus\,1,t\,\minus\,1]$) for~$q$, and thus satisfies (IV). Hence, by assumption, $T_{q}(\mu)$ is obtained from~$\mu$ by removing the minimal balanced rim hook starting at~$q+1$, if it exists and zero otherwise. There is an obvious one-to-one correspondence between the rim hooks starting at~$q+1$ for~$\lambda$ and~$\mu$, corresponding to `moving' the~$q+t+1$-box of the hook. This correspondence thus preserves the notion of balancedness. 

If $\lambda$ does not have a balanced rim hook starting at~$q$, we thus find $T_q(\lambda)=T_{q+t+1}T_q(\mu)=0$ since~$\mu$ satisfies (IV). If $\lambda$ does have a balanced rim hook starting at~$q$, then~$T_q(\mu)$ is obtained from~$\mu$ by removing its minimal rim balanced rim hook starting at~$q$. By construction~$T_q(\mu)$ then allows an addable~$q+t+1$-box and~$T_q(\lambda)=T_{q+t+1}T_q(\mu)$ is obtained by adding this box by (II). Hence also in this case, $\lambda$ satisfies indeed (IV).
\end{proof}

\begin{proof}[Proof of Proposition~\ref{PropUnique}]
By Lemma~\ref{Lem1}, (III) is satisfied. By Lemma~\ref{Lem2}, (IV) is satisfied for all partitions of types (d,$[r,0]$) and~(d,$[0,t]$). Lemmata~\ref{Lem3} and~\ref{Lem4} then allow to prove (IV) in general by induction on~$r$.

The proof of (V) is completely symmetrical to that of (IV).
\end{proof}

\subsection{Existence of the representation}
First we construct the Fock space representation.
\begin{lemma}\label{LemEasyRep}
We have a representation~$\Xi'$ of~$\TL_\infty(0)$ on~$\PA_\mZ$, determined by
$$T_q(\lambda)\;=\;\sum_{\mu}\bb^q_{\lambda\mu}\,\mu,\qquad\mbox{for all $\lambda\in\PA$}.$$
\end{lemma}
\begin{proof}
This is an easy combinatorial exercise, see also the proof of \cite[Proposition~2.3.1]{PB2}.
\end{proof}

We can identify the matrix $\cc$ with an automorphism of~$\PA_\mZ$, defined by
\begin{eqnarray}\label{eq:mat_to_endo}
\lambda\;\mapsto\;\sum_{\mu}\cc_{\lambda\mu}\mu.
\end{eqnarray}
Note that the above summation is finite, by~\ref{SecCell}. We twist the representation in Lemma~\ref{LemEasyRep} by this automorphism and use Proposition~\ref{Propaq}.
\begin{prop}
\label{PropExis}
The representation of~$\TL_\infty(0)$ on~$\PA_\mZ$ defined by
$$T_q(\nu)\;=\;\sum_{\lambda,\mu,\kappa}\cc_{\nu\lambda}\bb^q_{\lambda\mu}\cc^{\minus1}_{\mu\kappa}\,\kappa\;=\;\sum_\kappa \aac^q_{\nu\kappa}\kappa,$$
satisfies~$T_q(\varnothing)=\delta_{q0}\Box$ and~$T_q(\lambda)=\lambda\boxplus q$ if $\lambda$ has an addable~$q$-box.
\end{prop}

\begin{proof}
Using the elementary properties of~$\cc$ in \ref{SecCell} and the definition of~$\bb^q$, we find
$$T_q(\varnothing)\;=\;\sum_{\mu,\kappa}\bb^q_{\varnothing\mu}\cc^{\minus1}_{\mu\kappa}\,\kappa\;=\;\delta_{q,0}\sum_{\kappa}\cc^{\minus1}_{\Box\kappa}\kappa\;=\; \delta_{q,0}\Box.$$
To prove the second relation, we need to show that, for all partitions $\tilde{\lambda}=\lambda\boxplus q$, we have
\begin{eqnarray} \label{eq:two_sets}
\sum_{\mu}\cc_{\lambda\mu}\bb^q_{\mu\nu}\;=\; \cc_{\tilde{\lambda}\nu},\qquad\mbox{for all $\nu\in\PA$}.
\end{eqnarray}
By Corollary~\ref{Coraa1}, we have $(\cc\bb^q)_{\lambda\nu}\ge \cc_{\tilde{\lambda}\nu}$, so we focus on the inequality in the other direction.

We first reformulate \eqref{eq:two_sets} combinatorially. We will assume the reader is familiar with the set~$\Gamma_0$ of connected hooks and the set~$\Gamma$ of skew Young diagrams introduced in~\cite[Section~3.3]{PB2}, which describe the matrix $\cc$. Let~$\mathcal{S}_1(\lambda)$ denote the {\em multiset} of partitions~$\nu$ obtained by the following procedure, first take a partition~$\mu\subseteq\lambda$ such that~$\lambda / \mu \in \Gamma$, then either add a $q$-box to~$\mu$ or remove a $(q-1)$-box from~$\mu$ to obtain the partition~$\nu$. This multiset is linked to the left-hand side of~\eqref{eq:two_sets}. Concretely, each~$\nu\in\PA$ appears $(\cc\bb^q)_{\lambda\nu}$ times in~$\mathcal{S}_1(\lambda)$. Let~$\mathcal{S}_2(\lambda)$ denote the set of partitions~$\nu\subseteq\tilde{\lambda}$ such that~$\tilde{\lambda} / \nu \in \Gamma$. This describes the right-hand side of~\eqref{eq:two_sets}. First we will show that each element in~$\mathcal{S}_1(\lambda)$ is also an element in~$\mathcal{S}_2(\lambda)$ and then secondly that~$\mathcal{S}_1(\lambda)$ is actually a set. In conclusion, we have $\mathcal{S}_1(\lambda)\subseteq\mathcal{S}_2(\lambda)$ and hence~$(\cc\bb^q)_{\lambda\nu}\le \cc_{\tilde{\lambda}\nu}$, which thus implies the proposition.

We start with the following observation, which follows from immediate application of the properties of~$\Gamma$. Let~$\mu$ be a partition such that~$\lambda / \mu = \gamma \in \Gamma$, with decomposition~$\gamma = \gamma_1 \sqcup \ldots \sqcup \gamma_r$, such that each $\gamma_i$ is a disjoint union of connected rim hooks belonging to $\Gamma_0$ in the partition~$\lambda\backslash ( \gamma_1 \sqcup \ldots \sqcup\gamma_{i-1})$.
 Under the assumption that~$\lambda$ has an addable~$q$-box, there is a $k$, $1 \leq k \leq r$, such that~$\gamma_1, \ldots , \gamma_{k-1}$ all contain a $q$, $(q-1)$ and~$(q+1)$-box, while~$\gamma_k$ does not contain a $q$-box, and~$\gamma_{k+1},\ldots,\gamma_r$ contain no boxes with any of the three contents.

Each $\gamma_1,\ldots, \gamma_{k-1}$ will thus contain a shape of the form
$$\begin{minipage}{1cm}\;\young(ab,c)\end{minipage},$$
with~$a$ being a $q$-box. By swapping $a$ for the~$q$-box below $b$ and to the right of~$c$ we thus obtain~$\gamma_1^\prime,\ldots,\gamma_{k-1}^\prime\in\Gamma_0$, such that the skew Young diagram $\gamma_1^\prime\sqcup\ldots\sqcup\gamma_{k-1}^\prime\in \Gamma$ is removable from~$\tilde{\lambda}$. 

From now on, to avoid additional notation, we denote by~$a$, $b$ and~$c$ the boxes, in the same configuration as above with~$a$ being a $q$-box, that are on the rim of the partition~$\lambda \setminus (\gamma_1 \sqcup \ldots \sqcup \gamma_{k-1})$. Furthermore denote by~$d$ the~$q$-box directly below $b$ and to the right of~$c$. By construction we know that~$\gamma_k$ does not contain~$a$, but may contain any of the other two boxes, and that~$d$ is directly adjacent to~$\gamma_{k-1}^\prime$ (it was the box in~$\gamma_{k-1}$ that was swapped for another box to obtain~$\gamma_{k-1}^\prime$). We treat the three possible cases one by one.

\textit{Case 1: $\gamma_k$ contains neither $b$ nor~$c$.} In this case, we can always add a $q$-box to $\mu$ (the box $d$) and sometimes it is possible to remove a $q-1$-box from~$\mu$ (the box $c$).

(i) If $\nu\in\mathcal{S}_1(\lambda)$ is obtained by adding the box $d$ to~$\mu$, we set~$\gamma_j^\prime = \gamma_j$ for~$j \geq k$ and obtain~$\nu = \tilde{\lambda}\setminus (\gamma_1^\prime \sqcup \ldots \sqcup \gamma_{r}^\prime)$, so $\nu\in\mathcal{S}_2(\lambda)$.

(ii) If $\nu\in\mathcal{S}_1(\lambda)$ is obtained by removing the~$(q-1)$-box $c$ from~$\mu$, we define $\gamma_k^\prime$ as the union of~$\gamma_{k}$ and the boxes $c$ and~$d$. By construction~$\gamma'_k$ is either an element of~$\Gamma_0$ or the disjoint union of two elements of~$\Gamma_0$. Furthermore, we set~$\gamma'_j=\gamma_j$ for~$j > k$ and we have $\nu = \tilde{\lambda} \setminus (\gamma_1^\prime \sqcup \ldots \sqcup \gamma_{r}^\prime)\in \mathcal{S}_2(\lambda)$.

\textit{Case 2: $\gamma_k$ contains $c$ but not $b$.} In this case it is obvious that one cannot add the~$q$-box $d$ to~$\mu$ and one can also not remove the~$(q-1)$-box directly to the left of~$a$ from~$\mu$. Thus this case will not produce any elements in~$\mathcal{S}_1(\lambda)$.

\textit{Case 3: $\gamma_k$ contains $b$ but not $c$.} As in Case 2, it is not possible to add the~$q$-box $d$, but it can be possible to remove the~$(q-1)$-box $c$. In case that this is possible we add the two boxes $c$ and~$d$ to~$\gamma_k$ as in Case 1 above to obtain~$\gamma_k^\prime$ and set~$\gamma'_j=\gamma_j$ for~$j > k$. Thus  $\nu = \tilde{\lambda} \setminus (\gamma_1^\prime \sqcup \ldots \sqcup \gamma_{r}^\prime)$.

In this way we have realised every element of~$\mathcal{S}_1(\lambda)$ as an element of~$\mathcal{S}_2(\lambda)$. 

Now we prove that~$\mathcal{S}_1(\lambda)$ is in fact a set, by showing that each element of~$\mathcal{S}_2(\lambda)$ can only be created in at most one of the above ways from the construction in the definition of~$\mathcal{S}_1(\lambda)$.
For this, note that in the different cases we obtain the following:
\begin{itemize}
\item {\em Case 1(i):} for~$p\in\{q-1,q,q+1\}$, the skew diagram $\tilde{\lambda}/\nu$ contains $k\,\minus\, 1$ $p$-boxes.
\item {\em Case 1(ii):} for~$p\in\{q-1,q\}$, the skew diagram $\tilde{\lambda}/\nu$ contains $k$ $p$-boxes and $k\,\minus\, 1$ $q+1$-boxes.
\item {\em Case 3:} for~$p\in\{q-1,q,q+1\}$, the skew diagram $\tilde{\lambda}/\nu$ contains $k$ $p$-boxes.
\end{itemize} 
Clearly there is no overlap between 1(ii) and the other cases. 
To distinguish elements obtained from Case 1(i) and 3, we look at the unique hook $\alpha$ in~$\Gamma_0$, in the covering (see \cite[3.3]{PB2}) of~$\tilde{\lambda}/\nu\in\Gamma$, which contains the~$q-1$-box with minimal anticontent. In case 1(i), we have $\alpha\subset\gamma'_{k-1}$ and the fact that the connected hooks in~$\gamma_{k-1}$ must satisfy the D-condition in~\cite[Definition~3.3.4]{PB2} shows that~$\gamma_{k-1}$ and also $\alpha$ contains a $q-2$-box. In Case 3, we have $\alpha\subset \gamma_k'=\gamma_k\sqcup \{c,d\}$. Since the box $c$ was not contained in~$\lambda/\mu\supset \gamma_k$, neither was the~$q-2$-box left of~$c$. Hence~$\alpha$ does not contain that~$q-2$-box. The $q-2$-box below $c$ belongs to $\gamma_{k-1}'$, so also not to $\alpha$.  In conclusion, $\alpha$ is different for cases 1(i) and~(3). A fixed element of~$\cS_2(\lambda)$ can thus only be identified in at most one way with an element of~$\cS_1(\lambda)$.
\end{proof}

\subsection{A filtration of~$\Xi$}\label{SecFilt}
Recall the set~$\PA^{\ge k}$ and $\PA^k$ from~\eqref{SetsPAk}
\begin{prop}
The representation~$\Xi$ of~$\TL_\infty(0)$ on~$\PA_\mZ$ restricts to subrepresentations $\Xi^{\ge k}$ on~$\PA^{\ge k}_\mZ$ for each $k\in\mN$.
\end{prop}
We denote the composition factors of the above filtration by
\begin{equation}\label{eqsections}\Xi^k\,:\;\TL_\infty(0)\,\to\,\End_\mZ(\PA^k_\mZ).
\end{equation}
The proposition follows immediately from the following lemma.
\begin{lemma}\label{Preserve}
If $\lambda\in\PA^{\ge k}$, for some $k\in\mN$, and $\aac^q_{\lambda\kappa}\not=0$ for some $q\in\mZ$, then~$\kappa\in\PA^{\ge k}$.
\end{lemma}
\begin{proof}
By Propositions~\ref{PropExis} and~\ref{PropUnique}, $\kappa$ is obtained from~$\lambda$ either by adding a $q$-box (in which case~$\partial^k\subseteq \lambda\subset \kappa$) or by removing a rim hook as described in \ref{PropUnique}(IV) or (V). We restrict to the case~(IV) for simplicity. The rim hooks which are removed are balanced and minimal with that property. This means that the minimal anticontent of a box in the hook is attained by the~$q+1$-box, since otherwise one could construct a smaller balanced rim hook which ends at the box before the first one with strictly smaller anticontent.  

Assume first that the rim $q$-box of~$\lambda$ is inside $\partial^k$. It is then necessarily a removable box in~$\partial^k$ (in other words a box with maximal anticontent in~$\partial^k$). As the rim $q+1$-box has anticontent one higher than the~$q$-box, the above observation on the anticontent shows that no boxes in the rim hook are contained in~$\partial^k$. If the rim $q$-box already is not contained in~$\partial^k$ then the~$q$-box consequently has higher content than the ones in~$\partial^k$ and the same reasoning thus allows to conclude that no boxes in the rim hook are contained in~$\partial^k$. 
\end{proof}

\subsection{Faithfulness of~$\Xi$}\label{SecFaithful}
\begin{prop}\label{PropFaith}
The representations $\Xi$ and $\Xi'$ are faithful.
\end{prop}
Before we get to the proof we need some preparatory results.

\subsubsection{}For every sequence of integers $\underline{i}=(i_1,\ldots,i_r)$, we define the element $T_{\underline{i}}=T_{i_1} \cdots T_{i_r}$ in~$\TL_\infty(0)$. We also denote by~$\ell(\underline{i})=r$ the length of~$\underline{i}$. We multiply sequences of integers by concatenation and, for~$a\le b$, we write $[a,b]$ for the sequence~$(a,a+1,\ldots,b)$. Sequences
\begin{equation}\label{eqwint}
\underline{w} = [a_1,b_1] \cdot [a_2,b_2]\cdot \ldots\cdot [a_r,b_r],
\end{equation}
for some $r \ge 0$ such that~$a_1 > a_2 >  \ldots >a_r$ and $b_1> b_2 > \cdots  > b_r$, will be called \emph{fully commutative sequences}. We denote the set of such sequences by~$\mathsf{fcs}$. The segments $[a_j,b_j]$ will be called the \emph{intervals} of~$\underline{w}$.

\subsubsection{} The algebra~$\TL_\infty(0)$ admits {\em a basis} of the form $\{T_{\underline{w}}\;|\; \underline{w} \in\fcs \}.$
By~\cite[Proposition 1]{Fan}, a basis of~$\TL_\infty(0)$ is given by products of generators corresponding to fully commutative elements in~$\mathbb{S}_\infty$, see \cite[Section 1]{Fan} for a definition of fully commutative elements. Furthermore, by~\cite[Corollary 5.8]{Stem}, fully commutative elements have a normal form given by the elements in~$\mathsf{fcs}$. Such an expression for a fully commutative element in~$\mathbb{S}_\infty$ is unique, see \cite[Section 1.3]{Stem}. 

\subsubsection{} Consider the Fock space representation~$\Xi'$ of~$\TL_{\infty}(0)$. For $\underline{w}\in\fcs$ and $\lambda \in\PA$, we denote by~$\langle T_{\underline{w}}(\lambda)\rangle_m$, the part of the summation in~$\Xi'(T_{\underline{w}})(\lambda)$ of partitions of size $|\lambda|-\ell(\underline{w})$.

\begin{lemma}\label{NewLem} Consider arbitrary $\underline{w},\underline{v}$ in~$\fcs$.
\begin{enumerate}[(i)]
\item For an arbitrary $\lambda\in \PA$, we have either $\langle T_{\underline{w}}(\lambda)\rangle_m=0$ or $\langle T_{\underline{w}}(\lambda)\rangle_m$ is a partition.
\item There exists $\lambda\in \PA$, such that~$\langle T_{\underline{w}}(\lambda)\rangle_m\not=0$.
\item If $\ell(\underline{w})=\ell(\underline{v})$ and $\langle T_{\underline{w}}(\lambda)\rangle_m=\langle T_{\underline{v}}(\lambda)\rangle_m\not=0$ for some $\lambda\in\PA$, then $\underline{w}=\underline{v}$.
\end{enumerate}
\end{lemma}
\begin{proof}
We will assume $\underline{w}$ is of the form \eqref{eqwint} for the entire proof. Consider first the interval $[a_r,b_r]$. Then $T_{[a_r,b_r]}$ can remove $b_r-a_r+1$ boxes in~$\lambda\in\PA$ if and only if there is an $i\in\mZ_{\ge 1}$ such that
$$\lambda_i-i=b_r-1\qquad\mbox{and}\qquad \lambda_i-\lambda_{i+1}\ge b_r-a_r+1.$$
In this case, the unique partition~$\overline{\lambda}$ of size $|\lambda|-(b_r-a_r+1)$ in the summation~$T_{[a_r,b_r]}(\lambda)$ is obtained from $\lambda$ by removing $b_r-a_r+1$ boxes in row $i$. We can use the above argument on~$T_{[a_{r-1},b_{r-1}]}(\overline{\lambda})$. Moreover, since~$b_{r-1}>b_r$, it follows that the row from which boxes are removed in this step is strictly above the previous one.

It follows that the unique partition of~$|\lambda|-\ell(\underline{w})$ which can appear in~$T_{\underline{w}}\lambda$ is obtained by removing $b_j-a_j+1$ boxes in the unique row $k$ for which $\lambda_k=b_j+j-1$. This already proves part (i). Furthermore, since the number of boxes which are removed in each row reflects the lengths of the intervals of~$\underline{w}$ and the rows in which they are removed determines the values $b_j$, we obtain part~(iii).

Now we prove part (ii). Take $p\in\mN$ such that~$p\ge 2-a_r-r$. 
We define $\lambda\in\PA$ of length $p+r$, by setting
$$\begin{cases} \lambda_{l}&=\;\;p+1+b_1-1,\qquad\mbox{for }\;1\le l\le p,\\
\lambda_{p+i}&= \;\;p+i+b_{i}-1,\qquad\mbox{for }\;1\le i\le r. 
\end{cases}$$
That this is a partition follows from $\underline{w}\in\fcs$ and the definition of~$p$.
Clearly, by acting with~$T_{[a_r,b_r]}$ we can remove $b_r-a_r+1$ boxes in row $p+r$. As such we obtain a partition~$\overline{\lambda}$ with
$$\overline{\lambda}_{p+r-1}=\lambda_{p+r-1}=p+r+b_{r-1}-2\qquad\mbox{and}\qquad \overline{\lambda}_{p+r}=p+r+a_r-2.$$
In particular
$$\overline{\lambda}_{p+r-1}-\overline{\lambda}_{p+r}=b_{r-1}-a_r\ge b_{r-1}-a_{r-1}+1.$$
Hence, $\langle T_{[a_{r-1},b_{r-1}]}(\overline{\lambda})\rangle_m$ will again be non-zero and we can proceed iteratively.
\end{proof}

\begin{proof}[Proof of Proposition~\ref{PropFaith}]
Since the representation~$\Xi$ in Proposition~\ref{PropExis} is obtained from $\Xi'$ in Lemma~\ref{LemEasyRep} by applying an automorphism we will only prove faithfulness of the latter.

Fix an arbitrary element $x$ in~$\TL_\infty(0)$, written as $\sum_{k=1}^m r_k T_{\underline{w}^k}$, with~$r_k \in R$ and the~$\underline{w}^k\in\fcs$ distinct. Assume that~$\underline{w}:=\underline{w}^1$ has maximal $\ell(\underline{w})$. By Lemma~\ref{NewLem}(ii), there exists $\lambda\in\PA$ such that the summation~$T_{\underline{w}}\lambda$ contains a partition of size $|\lambda|-\ell(\underline{w})$, say $\nu$, with coefficient $1$. If $\ell(\underline{w}^j)<\ell(\underline{w})$, then clearly $T_{\underline{w}^j}\lambda$ will not contain~$\nu$, since all appearing partitions will be of strictly bigger size. Furthermore, if $\ell(\underline{w}^j)=\ell(\underline{w})$, Lemma~\ref{NewLem}(i) and~(iii) imply that~$T_{\underline{w}^j}\lambda$ does not contain~$\nu$ either. This proves that~$x(\lambda)\not=0$.
\end{proof}

\subsection{The Temperley-Lieb algebra as an enveloping algebra}
Consider the $\mZ$-Lie algebra~$\mathfrak{sl}_\infty$ with standard Chevalley generators $\{e_i,f_i\,|\,i\in\mZ\}$. The Fock space representation~$\Phi$ of~$\mathfrak{sl}(\infty)$ on~$\PA_{\mZ}$, see e.g.~\cite[Section~2.3]{Oded}, is clearly such that
$$\Phi(e_i+f_{i-1})\;=\; \Xi'(T_i),\;\quad\mbox{for all $i\in\mZ$.}$$
Let $\mathfrak{k}$ denote the Lie subalgebra of~$\mathfrak{sl}_\infty$ generated by~$\{e_i+f_{i-1}\,|\,i\in\mZ\}$. By construction, we have
$$\Phi(U(\mathfrak{k}))\;=\; \Xi'(\TL_{\infty}(0)),$$
 as subalgebras of~$\End_{\mZ}(\PA_{\mZ}).$ By Proposition~\ref{PropFaith}, we thus have
 $$\TL_{\infty}(0)\;\cong\;U(\mathfrak{k})/K$$
 with~$K$ the kernel of~$\Phi|_{U(\mathfrak{k})}$. 

\section{Main theorems}\label{SecMain}

\subsection{Thick tensor ideals and cells in the periplectic Deligne category}
\subsubsection{Thick tensor ideals}
A thick tensor ideal in a Krull-Schmidt monoidal (super)category~$\cC$ is a full subcategory~$\ideal$ which is
\begin{itemize}
\item {\em an ideal:} $X\otimes Y\in \ideal$, whenever $X\in \ideal$ or~$Y\in \ideal$;
\item {\em thick:} if $Z\in \ideal$ satisfies~$Z\cong X\oplus Y$, then $X,Y\in\ideal$.
\end{itemize}
For $\cC$ and $\ideal$ as above, {\em the monoidal supercategory $\cC/\ideal$} is defined as the quotient category of~$\cC$ with respect to all morphisms which factor through objects in~$\ideal$. 

\begin{rem}
The first condition simplifies for braided monoidal supercategories, such as~$\DA$. The second condition implies in particular that~$\ideal$ is {\em strictly full}. Sometimes it is imposed that~$\ideal$ must also be an {\em additive} subcategory. As all thick tensor ideals in~$\DA$, using the above definition, will be generated by one indecomposable object, they are obviously additive.
\end{rem}

Let~$\ideal_k$ denote the thick tensor ideal in~$\DA$ generated by~$R(\partial^{k})$. Concretely, $\ideal_k$ is the strictly full additive subcategory which contains all direct summands of~$R(\partial^k)\otimes R(\nu)$ for all $\nu\in\PA$.

\begin{thm}\label{Classif}
The set~$\{\ideal_k\,|\,k\in\mN\}$ yields a complete set of thick tensor ideals in~$\DA$. The indecomposable objects in~$\ideal_k$ are (up to isomorphism) given by
$\{R(\lambda)\;|\; \partial^{k}\subseteq\lambda\}.$ We thus have one chain of ideals
$$\DA\;=\;\ideal_0\;\supsetneq\; \ideal_1\;\supsetneq\;\cdots\;\supsetneq\; \ideal_k\;\supsetneq\;\ideal_{k+1}\;\supsetneq\;\cdots.$$
\end{thm}
\begin{proof}
Proposition~\ref{PropExis} implies that~$\TT(R(\nu))=R(\nu)\otimes R(\Box)=\oplus_\kappa R(\kappa)^{\oplus \aac_{\nu\kappa}}$ contains any $R(\kappa)$, with~$\kappa$ obtained by adding a box to $\nu$.
Consequently, $\ideal_k$ contains $R(\lambda)$ for all partitions~$\lambda$ which contain~$\partial^k$. On the other hand, Lemma~\ref{Preserve} implies that~$R(\lambda)\in\ideal_k$ requires $\partial^k\subseteq \lambda$.

It thus suffices to show that there are no more thick tensor ideals. Let~$\ideal$ be such an ideal and~$\partial^k$ the largest 2-core which is contained in all $\lambda$ with~$R(\lambda)\in\ideal$. 
Let~$\nu$ be a partition with~$R(\nu)\in\ideal$, with~$\partial^{k+1}\not\subset \nu$ and which has minimal $|\nu|$ under those two restrictions. 
Assume first that~$\nu\not=\partial^k$. Then~$\nu$ must contain a removable rim 2-hook and by 
Lemma~\ref{RemoveBox} there exists~$\kappa\subsetneq \nu$ such that~$R(\kappa)$ is a direct summand of~$\TT(R(\nu))$. This violates the minimality of~$|\nu|$, so $\nu=\partial^k$. By the above paragraph we then find $\ideal=\ideal_k$. 
\end{proof}

\subsubsection{Two-sided cells}\label{cells} Following~\cite[Section~3]{MM}, we have the notions of left, right and two-sided cells on a monoidal supercategory. As we work with symmetric categories, these three notions coincide. The quasi-order $\preceq$ on the set of isomorphism classes of indecomposable objects in a Krull-Schmidt symmetric monoidal (super)category~$\cC$ is determined by
$$[X]\;\preceq\; [Y],\qquad\mbox{if there exists $Z\in \Ob \cC$ such that~$Y$ is a direct summand of~$X\otimes Z$}.$$
There is a corresponding equivalence relation, defined as $[X]\sim [Y]$ if we have both $[X]\preceq [Y]$ and $[Y]\preceq [X]$. We denote the equivalence class of~$[X]$ under $\sim$ by~$[[X]]$.

For each $c=[[X]]$, we consider the additive strictly full subcategory $\cC_c$ generated by the indecomposable objects $Y\in\Ob\cC$ with~$[Y]\ge [X]$. This is the thick tensor ideal in~$\cC$ generated by~$X$. Furthermore, we have the additive strictly full subcategory $\overline{\cC}_c$ of~$\cC_c$ corresponding to the indecomposable objects $Y\in\Ob\cC$ with~$[Y]\not\sim [X]$.
The {\em cells} of~$\cC$ are the quotient categories~$\cC_c/\overline{\cC}_{c}$. We call a cell {\em maximal} if it corresponds to indecomposable objects which are maximal in the quasi-order. A maximal cell hence corresponds to a subcategory of~$\cC$.

Recall the subsets of~$\PA$ in \eqref{SetsPAk}. Clearly, for~$\DA$, the quasi-order $\preceq$ is total:
$$[R(\lambda)]\preceq [R(\mu)]\quad\mbox{if and only if $k\le l$, with~$\lambda\in\PA^k$ and $\mu\in\PA^l$.}$$
\begin{cor}
The set~$\{\ideal_k/\ideal_{k+1}\,|\,k\in\mN\}$ yields a complete set of cells in~$\DA$.
\end{cor}
We clearly have $[\ideal_k/\ideal_{k+1}]_{\oplus}\;\cong\; \PA^k_{\mZ}$.



\subsection{The $\Ob$-kernel of the universal tensor functor}

By \cite[Section~5]{Kujawa} (see also~\cite[Section~4.5]{Vera}), for any $n\in\mZ_{\ge 1}$, we have a monoidal superfunctor
\begin{equation}\label{FunF}F_n:\;\DA\,\to\, \Fn,\end{equation}
where $i\in\mN\subset\Ob\DA$ gets mapped to~$V^{\otimes i}$ and $\cup\in\Hom_{\cA}(0,2)$ is mapped to
$$F_n(\cup)\in\Hom_{\Fn}(\mk,V^{\otimes 2})=\Hom_{\pe(n)}(V^{\otimes 2},\mk),\quad\mbox{given by }\quad\Fn(\cup)(v\otimes w)=\beta(v,w),$$
with~$\beta$ the defining bilinear form in Section~\ref{Secpn}.
In particular, $F_n$ induces the algebra morphisms
$$\phi_n^r\,: \; A_r\;\to\;\End_{\pe(n)}(V^{\otimes r})^{\op},$$
first introduced in~\cite[Proposition~2.4]{Moon}. 

\begin{thm}\label{ThmFunctor}
The monoidal superfunctor \eqref{FunF}
is full and its $\Ob$-kernel is given by~$\ideal_{n+1}$.
\end{thm}

We start with two preparatory lemmata.

\begin{lemma}\label{TrivLem}
For $\lambda\in\PA$, we have $F_n(R(\lambda))=0$ if and only if $\phi^r_n(e)=0$ for an arbitrary~$r\in\mN$ with~$|\lambda|\in\JJJ^0(r)$ and~$e\in A_r$ an idempotent corresponding to~$L_r(\lambda)$.
\end{lemma}
\begin{proof}
By Remark~\ref{RemThm}, we have $R(\lambda)\cong (r,e)$ in~$\DA$. Furthermore, by definition of~$\phi^r_n$, we have
$F_n((r,e))\;=\; \im\phi^r_n(e).$ This concludes the proof.
\end{proof}

\begin{lemma}\label{LemKernel}
For any partition~$\lambda$ with~$\lambda_{n+1}>n$, we have $F_n(R(\lambda))=0$.
\end{lemma}
\begin{proof}
By Lemma~\ref{TrivLem}, it suffices to prove that~$\phi^r_n(e_\lambda)=0$, with~$r=|\lambda|$.

When we restrict the action~$\phi^r_n$ from~$A_r$ to the subalgebra~$\mk\mS_r$ (see~\ref{defAr}), the image commutes with the~$\mathfrak{gl}(n|n)$ action on~$V^{\otimes r}$, see~\cite[Theorem~4.14]{Berele}. Hence, we have a commuting diagram:
\[
\xymatrix{
 *+\txt{$A_r$} \ar[rr]^{\phi^r_n}&& *+\txt{$\End_{\mathfrak{pe}(n)}(V^{\otimes r})^{\op}$} \\
*+\txt{$\mk\mS_r$}\ar[rr]^{}\ar@{^{(}->}[u]&&*+\txt{$\End_{\mathfrak{gl}(n|n)}(V^{\otimes r})^{\op}$.}\ar@{^{(}->}[u]
}
\]

Now, let~$f_\lambda\in\mk\mS_r$ be a primitive idempotent corresponding to the (simple) Specht module for~$\lambda$. By the choice of the labelling of simple modules over $A_r$, $e_\lambda$ appears (up to conjugation) in the decomposition of~$f_\lambda$ into primitive idempotents in~$A_r$, see e.g.~\cite[Corollary~4.3.3]{PB1}.

The hook condition in~\cite[Theorem~3.20]{Berele} and the above commuting diagram together imply that~$\phi_n^r(f_\lambda)=0$ if $\lambda_{n+1}>n$, so in particular $\phi_n^r(e_\lambda)=0$. 
\end{proof}

\begin{proof}[Proof of Theorem~\ref{ThmFunctor}]
By additivity it suffices to show that the restriction~$\cA\to \Fn$ is full. By~\cite[Section~5.3]{Kujawa}, the
surjectivity of
$$\Hom_{\cA}(i,j)\;\to\; \Hom_{\pe(n)}(V^{\otimes j},V^{\otimes i}),$$
for any $i,j\in\mN$, is equivalent to surjectivity of 
$$\Hom_{\cA}(0,i+j)\;\to\; \Hom_{\pe(n)}(V^{\otimes (j+i)},\mk).$$
The latter is precisely~\cite[Section~4.9]{DLZ}.

As $F_n$ is a monoidal superfunctor, its Ob-kernel $\KK_n$ is a thick tensor ideal. By Theorem~\ref{Classif}, we thus have $\KK_n=\ideal_k$ for some $k\in\mN$.

By \cite[Corollary~8.2.7]{PB1}, for~$\lambda\vdash r$ with~$\lambda_{n+1}=0$, we have $\phi_n^r(e_\lambda)\not=0$. By Lemma~\ref{TrivLem}, we find that in particular $R(\partial^{n})\not\in\KK_n$, which implies~$\KK_n\not=\ideal_k$ when~$k\le n$.

By Lemma~\ref{LemKernel}, we have $R(\lambda)\in\KK_n$ for~$\lambda=(n+1,\ldots,n+1)$, the partition of~$(n+1)^2$ of length $n+1$. As $\partial^{k}\not\subseteq\lambda$ for~$k>n+1$, we find $\KK_n\not=\ideal_k$ when~$k> n+1$. This concludes the proof.
\end{proof}

\subsection{Tensor powers of the natural representation of~$\mathfrak{pe}(n)$}\label{SecTenProj} The results in the previous subsection allow to classify the indecomposable summands in the~$\mathfrak{pe}(n)$-module~$V^{\otimes r}$ up to isomorphism. In this subsection we further determine when the direct summands are projective.

\begin{thm}\label{ThmProj}
The assignment 
$$\lambda\mapsto R_{n}(\lambda):=F_n(R(\lambda)),$$
is a bijection between $\PA^{\le n}$ and the set of isomorphism classes of indecomposable summands in~$\bigoplus_{r\in\mN}V^{\otimes r}$. The module $F_n(R_n(\lambda))$ appears as a direct summand in~$V^{\otimes r}$ if $|\lambda|\in\JJJ^0(r)$ and is projective if and only if $\lambda\in\PA^n$.
\end{thm}

We denote the full subcategory of projective modules in~$\Fn$ by~$\mathfrak{pe}(n)$-proj.

\begin{thm}\label{ThmProj2}
The subcategory~$\pe(n)${\rm -proj} is the unique maximal cell in~$\Fn$. The functor~$F_n$ restricts to an essentially surjective functor~$\ideal_n\to \pe(n)${\em-proj} with Ob-kernel $\ideal_{n+1}$. Hence, there exists a superfunctor
$$\ideal_n/\ideal_{n+1}\;\to\; \pe(n)\mbox{{\rm-proj},}$$
which is essentially bijective and full.
\end{thm}

\begin{rem}
It will be proved in \cite{Ideals} that this superfunctor is actually an equivalence.
\end{rem}

Note first that since the tensor functor $F_n$ is full, it maps indecomposable objects in~$\DA$ to objects in~$\Fn$ which are indecomposable or zero. We use this fact freely.

Now we prove these two theorems.
There is a duality $\ast$ on~$\Fn$, see \cite[Section~2]{gang}. Furthermore, the right adjoint of~$-\otimes M$ is $-\otimes M^\ast$, for a module~$M$, see e.g. \cite[Section~4.4]{gang}. This implies that~$M\otimes N$ is projective as soon as either $M$ or~$N$ is projective. Consequently, $\mathfrak{pe}(n)$-proj is a thick tensor ideal.

\begin{lemma}\label{LemProj}
Let~$Q_1,Q_2$ be arbitrary indecomposable projective modules in~$\Fn$. Then $Q_1$ is a direct summand of~$Q_2\otimes V^{\otimes k}$ for some $k\in\mN$.
\end{lemma}
\begin{proof}
It is well-known that injective and projective modules coincide in~$\Fn$, see e.g.~\cite{gang, Chen}. In particular, the duality $\ast$ maps projective modules to projective modules. We then find that there exists $j,i\in\mN$ such that~$Q_1$ is a direct summand of~$V^{\otimes j}$ and~$Q_2^\ast$ is a direct summand in~$V^{\otimes i}$, by \cite[Lemma~8.3.2]{PB1}. Then we have a composition of epimorphisms
$$Q_2\otimes V^{\otimes i+j}\tto Q_2\otimes Q_2^\ast\otimes Q_1\tto Q_1.$$ 
Since~$Q_1$ is projective, the corresponding epimorphism splits, which concludes the proof.
\end{proof}

\begin{cor}\label{ProjCell}
The full subcategory~$\pe(n)${\rm -proj} is the unique maximal cell in~$\Fn$.
\end{cor}
\begin{proof}
Lemma~\ref{LemProj} implies that all indecomposable projective modules are equivalent for the relation of \ref{cells}. This implies that the thick tensor ideal is in fact a cell. Moreover, since~$V^{\otimes k}$ contains a projective direct summand for~$k>>$ (\cite[Lemma~8.3.2]{PB1}), $M\otimes V^{\otimes k}$ for any module~$M$ contains a projective direct summand, showing that~$\pe(n)${\rm -proj} is the unique maximal cell.\end{proof}

\begin{proof}[Proof of Theorem~\ref{ThmProj}]
The classification of indecomposable summands is an immediate consequence of Theorem~\ref{ThmFunctor} and Lemma~\ref{TrivLem}. 

The projective modules form a thick tensor ideal in~$\Fn$. This implies that the corresponding pre-image under the tensor superfunctor~$F_n$ also forms a thick tensor ideal $\ideal$ in~$\DA$. By Theorem~\ref{ThmFunctor}, we have $\ideal_{n+1}\subseteq\ideal$. Since each projective module appears as direct summands of~$V^{\otimes j}$ for some $j\in\mN$ (\cite[Lemma~8.3.2]{PB1}), we even have~$\ideal_{n+1}\subsetneq \ideal$. By Theorem~\ref{Classif}, we thus have $\ideal=\ideal_k$ for some $k\le n$. 
Because $F_n$ is full (Theorem~\ref{ThmFunctor}), Corollary~\ref{ProjCell} implies that~$\ideal_k/\ideal_{n+1}$ must be a cell, so $k=n$.
\end{proof}

\begin{proof}[Proof of Theorem~\ref{ThmProj2}]
This follows from Corollary~\ref{ProjCell} and Theorems~\ref{ThmFunctor} and~\ref{ThmProj}.
\end{proof}

\subsection{A bijection~$\Par\to \mP(\mZ)$}

\subsubsection{}\label{mark} For an arbitrary partition~$\lambda$, we ``mark'' its Young diagram as follows. We start by putting a diamond $\diamond$ in the right-most box of the bottom row. Then we move up the rows as follows. In any given row we put a $\diamond$ in the right-most box if the number of $\diamond$ we have added so far is strictly smaller than the column of the box in question. Note that the column of that box is precisely $\lambda_i$, with $i$ the row considered.

\begin{ex} The following are Young diagrams marked according to the procedure in \ref{mark}.
\begin{gather*}
\young(\leeg \diamond),\qquad\young(\leeg,\leeg,\diamond),\qquad \young(\leeg\leeg,\leeg \diamond,\leeg \diamond),\qquad \young(\leeg\leeg,\leeg\diamond,\leeg,\diamond),\qquad \young(\leeg\leeg \diamond,\leeg\leeg,\leeg \diamond,\leeg \diamond),\qquad \young(\leeg\leeg\leeg \diamond,\leeg \diamond,\diamond).
\end{gather*}
\end{ex}

\begin{lemma}
For $\lambda\in\Par^n$, we have precisely $n$ diamonds in the marking.
\end{lemma}
\begin{proof}
Assume that we have precisely $n$ $\diamond$ in the marking of $\lambda\in\Par$. Take the marked Young diagram obtained from $\lambda$ by removing all rows without $\diamond$. Denote the corresponding partition by~$\mu$. By construction, the $\diamond$ in the diagram corresponds to the marking for $\mu$ as in~\ref{mark}. Now we have a partition~$\mu$ with $n$ rows and a $\diamond$ in each row. This means that $\mu_i\ge n-i+1$, for all $1\le i\le n$. Hence, we have $\partial^n\subset\mu\subset\lambda$.

To conclude the proof it suffices to show that $\partial^n\subset\lambda$ implies that we have at least $n$ $\diamond$ in the marking of $\lambda$. We do this by induction on~$n$. For $n\in\{0,1\}$, this is trivial. Now take~$\lambda\in\Par^{\ge n}$ and denote by~$\nu$ the partition obtained by deleting the first row in~$\lambda$. By construction, $\nu\in \Par^{\ge n-1}$. If $\nu$ has at least $n$ boxes in its marking, then clearly so does $\lambda$. By the induction hypothesis, the only other option is that $\nu$ has precisely $n-1$ $\diamond$. To obtain the marking in~$\lambda$ we just take the one in~$\nu$ and need to decide whether a $\diamond$ needs to be added in the upper row in~$\lambda$. Since $\lambda_1\ge n$ and we have only added $n-1$ boxes so far, this is always the case. Hence the marking in~$\lambda$ always contains at least $n$ $\diamond$.
\end{proof}


\begin{ddef}\label{Defdlam}
For any $\lambda\in \Par$, we denote by~$\widetilde{d}_\lambda\subset\mZ$ the set with as elements the contents of the boxes in~$\lambda$ which contain a $\diamond$ according to the marking in \ref{mark}. Denote by~$d_\lambda$ the set obtained from $\widetilde{d}_\lambda$ by subtracting $1$ from each element.
\end{ddef}

We can consider the corresponding maps $d: \Par^n\to \mP(\mZ;n)$, for all $n\in\mN$, recall here the notation $\mP(\mZ;n)$ from Section \ref{SecPrel}. It will follow from~\ref{bij} that this is actually a bijection. Consequently, we find that $d$ is a bijection
$$d:\,\Par\,\stackrel{\sim}{\to}\,\mP(\mZ),\quad \lambda\mapsto d_\lambda.$$

\begin{lemma}\label{Lemaddq}
Consider $q\in\mZ$ and $\lambda\in\Par^n$ with addable $q$-box, with $\mu:=\lambda\boxplus q\in\Par^n$. Then we have
$\TT_q(R(\lambda))\cong R(\mu)$, and
\begin{enumerate}[(i)]
\item if $\lambda$ has a marked box with content $q-1$, then $d_\mu$ is obtained from $d_\lambda$ by replacing $q-2\in d_\lambda$ by $q-1$.
\item if $\lambda$ has a marked box with content $q+1$, but no marked box with content $q-1$, then $d_\mu$ is obtained from $d_\lambda$ by replacing $q\in d_\lambda$ by $q-1$.
\end{enumerate}
\end{lemma}
\begin{proof}
That $\TT_q(R(\lambda))\cong R(\mu)$ is Proposition~\ref{PropExis}. Part (i) then follows immediately from the definition of the marking. For part (ii), we first observe that the assumptions imply that the box above the $q+1$-box with a $\diamond$ does not contain a $\diamond$. Indeed, existence of such a $\diamond$ in the marking would imply that the box with content $q-1$ on the rim should have a $\diamond$ too. The claim then follows again from the definition of the marking. 
\end{proof}

\subsection{Link between the labelling sets}
Fix $n\in\mN$. By Theorem~\ref{ThmProj2}, the superfunctor~$F_n$ induces a bijection between $\{R(\lambda)\,|\,\lambda\in \PA^n\}$ and indecomposable projective modules $\{P(\omega)\,|\, \omega\in X_n^+\}$ in~$\Fn$, with notation as explained below. Now we will describe this bijection.

\subsubsection{} The projective module $P(\omega)$ is labelled by the highest weight $\omega$ of its simple top. 
We follow the conventions of \cite[Section~2]{gang} regarding root system and notation of weights. The set of integral dominant weights is given by
$$X^+_n\;:=\;\{\omega=\sum_{i=1}^n\omega_i\varepsilon_i\;|\; \omega_j\in\mZ\mbox{ and }\omega_1\ge \omega_2\ge\cdots\ge \omega_n\}.$$
As in \cite[Section~2.2]{gang}, we introduce the bijection
$$c:\,X^+_n\,\stackrel{\sim}{\to}\, \mP(\mZ;n),\;\, \omega \mapsto c_\omega=\{\omega_1+n-1,\omega_{2}+n-2,\ldots,\omega_n\}.$$

\begin{ddef}
The map $\mathbf{f}$ is given by
$$\mathbf{f}:\,\Par^n\;\to\;X_n^+,\;\; \lambda\mapsto c^{-1}(d_\lambda),$$
with $d_\lambda\subset\mZ$ as in Definition~\ref{Defdlam}.
\end{ddef}

\begin{thm}\label{ThmEl}
For all $\lambda\in\Par^n$, we have
$F_n(R(\lambda))\;\cong \;P(\mathbf{f}(\lambda)).$
\end{thm}

\subsubsection{}\label{bij}
Before proceeding to the proof of Theorem~\ref{ThmEl}, we observe that comparison with Theorem~\ref{ThmProj2} implies that $\mathbf{f}$ is actually a bijection. Consequently, 
$d: \Par^n\to\mP(\mZ;n)$ must also be a bijection.

\subsubsection{}\label{k0}
If $\lambda\in\PA^n$, there is at least one $0\le k_0\le n$ such that $\lambda_{k_0+1}= n-k_0$. Note that, since~$\lambda\in\PA^{\ge n}$, we have $\lambda_i\ge n+1-i$, for all $1\le i\le n$. For $\lambda\in\PA^n$, with such $k_0$, we define~$\nu\in\PA$ by 
$$\nu^t=(\lambda_{k_0+1},\lambda_{k_0+2},\ldots).$$ 
As a special case of Theorem~\ref{ThmEl}, we have the following closed formula for generic~$F_n(R(\lambda))$.

\begin{prop}\label{PropLink}
Take $\lambda\in \PA^n$, with $k_0$ and $\nu$ as in \ref{k0}. Assume that $\nu_i\not=\nu_j$, whenever~$i\not=j$. Then we have
$$F_n(R(\lambda))\;\cong\;P\left(\sum_{i=1}^{k_0}(\lambda_i-n-1)\varepsilon_i+\sum_{i=k_0+1}^n(-\nu_{n-i+1}-k_0)\varepsilon_i\right).$$
\end{prop}

Now we proceed to the proof of Theorem~\ref{ThmEl}.
Consider the exact functor $\Theta=-\otimes V$ on~$\Fn$, with~$V$ the natural $\mathfrak{pe}(n)$-module, as introduced in \cite[Section~4.1]{gang}. This functor has a natural transformation~$\Omega$, see \cite[Lemma~4.1.4]{gang}, according to which we have a decomposition~$\Theta=\bigoplus_{j\in\mZ}\Theta_j$, see \cite[Proposition~4.1.9]{gang}.
\begin{lemma}\label{LemTTTheta}
We have a natural isomorphism $F_n \circ \TT_j\stackrel{\sim}{\Rightarrow} \Theta_j\circ F_n$, for all $j\in\mZ$.
\end{lemma}
\begin{proof}
Since $F_n(R(\Box))\cong V$, we clearly have $F_n\circ \TT\cong \Theta\circ F_n$.
Consider $(r,e)\in \DA$. By \cite[8.5.3]{PB1}, the morphism $\xi_{(r,e)}=ex_{r+1}e$ of $(r,e)$ is mapped to $\Omega_{F_n (r,e)}$. The result then follows easily.
\end{proof}

We will use the above connection between the functors $\TT_j$ and $\Theta_j$ freely.

\begin{lemma}\label{Lem2core}
We have
$$F_n(R(\partial^n))\cong P(-\varepsilon_1-2\varepsilon_2-\cdots -n\varepsilon_n).$$
\end{lemma}
\begin{proof}
Assume $P(\omega)\cong F_n(R(\partial^n))$, for $\omega=\sum_{i=1}^n\omega_i\varepsilon_i$. We have
$$\TT_n(R(\partial^n))\not=0,\qquad\mbox{and}\qquad \TT_k(R(\partial^n))=0\quad\mbox{for all $k>n$}.$$
By \cite[Section~7.2]{gang}, this means that $\max c_{\omega}=n-2$, so $\omega_1=-1$. 
Similarly,
$$\TT_{-n}(R(\partial^n))\not=0,\qquad\mbox{and}\qquad \TT_{-k}(R(\partial^n))=0\quad\mbox{for all $k>n$},$$
implies that  $\omega_n=\min c_{\omega}=-n$.

Now assume that we would have $\omega_i=\omega_{i+1}$ for $1\le i<n$. This means we would have two integers in~$c_{\omega}$ which differ by one. Since $\max c_{\omega}-\min c_{\omega}=2n-2$, this would imply that there exists $a\in\mZ$ such that $\{a,a+1\}\subset c_{\omega}$, but $a+2\not\in c_{\omega}$. By \cite[Lemmata 7.2.1(1)]{gang}, $\TT_{a+3}R(\partial^n)$ is non-zero. By \cite[Lemma~7.2.3(1)]{gang}, $\TT_{a+2}R(\partial^n)$ is non-zero. Since there is no $a\in\mZ$ for which both these statements are true, the only remaining option for $\omega$ is the one proposed in the lemma.
\end{proof}

\begin{proof}[Proof of Theorem~\ref{ThmEl}]
The claim is true for $\lambda=\partial^n$, by Lemma~\ref{Lem2core}. Now take an arbitrary $\lambda\in\Par^n$, with $k_0$ as in \ref{k0}. We `construct' $\lambda$ from $\partial^n$ in two steps. In step (a) we add $\lambda_1-n$ boxes in row 1, then $\lambda_2-n+1$ boxes in row 2, and so on, until we add $\lambda_{k_0}-n+k_0-1$ boxes in row $k_0$. In step (b) we similarly add $\lambda_1^t-n$ boxes in column 1, then boxes in column 2, until we conclude by adding sufficiently many boxes in column $n-k_0$.

Let $\mu\in\Par^n$ be a partition obtained while constructing $\lambda$ as above and let $\mu'\in\Par^n$ be the partition obtained by adding one more box to $\mu$ along the procedure towards $\lambda$. Then we have $\partial^n \subset \mu \subset \mu' \subset \lambda$. Assume that the claim holds for $\mu$.

If $\mu'$ is obtained from $\mu$ by adding a box as in step (a), then, by construction, the row in which a box is added to obtain $\mu'$ contains a $\diamond$. The change $d_{\mu}\mapsto d_{\mu'}$ is thus as in Lemma \ref{Lemaddq}(i). Comparing this with \cite[Lemma 7.2.1(1)]{gang} gives the claim for $\mu'$.

If $\mu'$ is obtained from $\mu$ by a construction step of type (b), then both cases in Lemma \ref{Lemaddq} can occur. In case (i), comparing again with \cite[Lemma 7.2.1(1)]{gang} gives the claim for $\mu'$. In case (ii), comparing with \cite[Lemma 7.2.3(2)]{gang} gives the claim.

Schematically, we can represent the relation between the local combinatorics of marked partitions of \ref{mark} and weight diagrams of \cite[Section~5.1]{gang}, under the above construction, as follows
$$\young(\leeg ?,\diamond)\;\quad\leftrightarrow\;\quad \stackrel{\bullet}{i-2}\;\;\stackrel{\circ}{i-1}\;\;\;\;\stackrel{?}{i}\,,\qquad \young(\leeg \diamond,\leeg)\;\quad\leftrightarrow\;\quad \stackrel{\circ}{i-2}\;\;\stackrel{\circ}{i-1}\;\;\;\;\stackrel{\bullet}{i}.$$
Here the content of the addable box is $i$, and $\young(?)$ is a box which may or may not contain $\diamond$.
\end{proof}


\section{Categorification of the representation~$\Xi$}
\label{SecCateg}
\subsection{Categorical action of~$\TL_\infty(0)$ on~$\DA$}
In this section we upgrade the na\"ive categorification~$(\DA,\Psi,\{\TT_i\,|\,i\in\mZ\})$ from Theorem~\ref{ThmDecat} to a 
 `weak categorification' of the~$\TL_{\infty}(0)$-representation~$\Xi$, in the sense of \cite[Definition~2.7]{Maz}. 
\begin{thm}\label{ThmCat}
We have natural isomorphisms of functors, for all $i,j\in\mZ$ with~$|i-j|>1$,
$$\TT_i^2\;\stackrel{\sim}{\Rightarrow}\; 0,\;\TT_i\TT_j\;\stackrel{\sim}{\Rightarrow}\; \TT_j\TT_i\;\;\mbox{and}\;\;\; \TT_i\TT_{i\pm 1}\TT_i\;\stackrel{\sim}{\Rightarrow}\; \TT_i.$$
The second natural isomorphism is even, the third one odd.
\end{thm}

\begin{rem}\label{RemCat}
Theorem~\ref{ThmCat} implies also that we get a weak categorification of the~$\TL_\infty(0)$-representation~$\Xi^k$ in \eqref{eqsections} on the cell $\ideal_k/\ideal_{k+1}$.
\end{rem}
To stress the similarity with the notion of $\fg$-categorification for a Kac-Moody algebra $\fg$ in \cite[Definition~5.29]{Rouq}, we add the following proposition. For this we use the affine periplectic Brauer algebra $\widehat{P}_d^-$, with generators $s_i,\epsilon_i$, for $1\le i<d$, and $y_j$, for $1\le j\le d$, and relations as in \cite[Definition~3.1]{ChenPeng}.
\begin{prop}\label{PropAff}
We have even natural transformations $\sigma,\tau:\TT\TT\Rightarrow \TT\TT$, such that for any $d\in\mN$, we get an algebra morphism
$$\widehat{P}_d^-\;\to\;\Nat(\TT^d,\TT^d)_{\oa},\qquad\mbox{given by }$$
$$s_i\mapsto \TT^{d-i-1}(\sigma_{\TT^{i-1}})\;\,\mbox{ and }\;\,\epsilon_i\mapsto \TT^{d-i-1}(\tau_{\TT^{i-1}}),\quad\mbox{for $1\le i<d$},$$
$$y_j\mapsto \TT^{d-j}(\xi_{\TT^{j-1}}),\quad\mbox{for $1\le j\le d$,}$$
with $\xi:\TT\Rightarrow \TT$ as in  Section~\ref{SecXi}.
\end{prop}

\subsubsection{}\label{Defepsilon}We start the proofs of the above results by defining two families of odd morphisms. For~$X=(r,e)\in\DA$, we set, using \eqref{MorDeli},
$$\varepsilon_X:\TT\TT(X)\to X,\quad \varepsilon_X=(e\otimes\cap)\in e\Hom_{\cA}(r+2,r)(e\otimes\II\otimes\II),\quad\mbox{and}$$
$$\eta_X:X\to \TT\TT(X),\quad \eta_X=(e\otimes\cup)\in (e\otimes\II\otimes\II)\Hom_{\cA}(r,r+2)e.$$
These extend easily to arbitrary objects~$X\in\Ob\DA$.
\begin{lemma}\label{LemCat1}
The families~$\{\varepsilon_X\,|\,X\in\Ob\DA\}$ and~$\{\eta_X\,|\, X\in\Ob\DA\}$ define odd natural transformations $\varepsilon:\TT\TT\Rightarrow \Id$ and~$\eta:\Id\Rightarrow\TT\TT$.
\end{lemma}
\begin{proof}We prove the claim for~$\varepsilon$, the case~$\eta$ is proved identically.
Take idempotents~$e\in A_r$ and~$f\in A_s$ and set~$X=(r,e)$ and~$Y=(s,f)$. Consider $a\in f\Hom_{\cA}(r,s)e=\Hom_{\DA}((r,e),(s,f))$. By definition, we need to show that
$$a\circ (e\otimes\cap)\;=\; (-1)^{|a|}(f\otimes \cap)\circ (a\otimes\II\otimes\II).$$
This equation holds true by equation~\eqref{superComp} and $ae=a=fa$.
\end{proof}

\begin{lemma}\label{LemCat2}
We have equalities of natural transformations
$$\TT(\varepsilon)\circ\eta_{\TT}=1_{\TT}\qquad\mbox{and}\qquad \varepsilon_{\TT}\circ\TT(\eta) =-1_{\TT}.$$
\end{lemma}
\begin{proof}
By \cite[Theorem~3.2.1]{Kujawa}, we have
$$(\cap\otimes\II)\circ (\II\otimes\cup)=\II\quad\mbox{and }\quad(\II\otimes \cap)\circ(\cup\otimes\II)=-\II.$$
Set~$X=(r,e)\in\Ob\DA$. By \ref{Defepsilon}, equation~\eqref{eqDefT} and the above formula, we have
$$\TT(\varepsilon_X)\circ \eta_{\TT X}\;=\; (e\otimes \cap\otimes\II)\circ (e\otimes\II\otimes\cup)\;=\;(e\otimes\II)\;=\; 1_{\TT X}.$$
The second relation follows identically.
\end{proof}

\subsubsection{}\label{Defiotapi} We introduce natural transformations $\iota^i:\TT_i\Rightarrow \TT$ and~$\pi^i:\TT\Rightarrow \TT_i$. For~$X=(r,e)$, the morphism $\iota^i_X$, resp. $\pi^i_X$, is to be identified with
$$(e\otimes\II)\gamma^{(r+1)}_i\;=\;\gamma^{(r+1)}_i(e\otimes\II)\gamma^{(r+1)}_i\;=\;\gamma^{(r+1)}_i(e\otimes\II),$$
which can be interpreted inside $\Hom_{\DA}(\TT_iX,\TT X)$, resp. $\Hom_{\DA}(\TT X,\TT_i X)$, as in \eqref{MorDeli}. Furthermore, $(\iota^i\circ\pi^i)_X\in \Hom_{\DA}(\TT X,\TT  X)$ and~$1_{\TT_i X}=(\pi^i\circ\iota^i)_X\in \Hom_{\DA}(\TT_i X,\TT_i X)$ can also be interpreted as the above element. All this extends to arbitrary $X\in\Ob\DA$.

\begin{lemma}\label{LemCat3}
We have equalities of natural transformations, for all $i\in\mZ$,
$$\varepsilon\circ \iota^{i}_{\TT}\;=\; \varepsilon\circ (\iota^{i}\star (\iota^{i+1}\circ\pi^{i+1})),\qquad\mbox{for}\quad \TT_{i}\TT\Rightarrow\Id,$$
$$ \varepsilon\circ \TT(\iota^{i})\;=\;\varepsilon \circ ((\iota^{i-1}\circ\pi^{i-1})\star \iota^{i}),\qquad\mbox{for}\quad \TT\TT_{i}\Rightarrow\Id,$$
$$\TT(\pi^i)\circ\eta\;=\; ((\iota^{i+1}\circ\pi^{i+1})\star \pi^{i})\circ \eta,\qquad\mbox{for}\quad\Id\Rightarrow\TT\TT_i,$$
$$ \pi^{i}_{\TT}\circ\eta\;=\;(\pi^{i}\star(\iota^{i-1}\circ\pi^{i-1}))\circ\eta,\qquad\mbox{for}\quad\Id\Rightarrow\TT_{i}\TT.$$
\end{lemma}
\begin{proof}
By \cite[Lemma~6.3.1(1)]{PB1}, for $k\in\mN$, we have
$$(\II^{\otimes k}\otimes \cap)\circ x_{k+1}\;=\;(\II^{\otimes k}\otimes \cap)\circ (x_{k+2}+\II^{\otimes k+2})\quad\mbox{and}$$
$$x_{k+2}\circ(\II^{\otimes k}\otimes \cup)\;=\; (x_{k+1}+\II^{\otimes k+2})\circ (\II^{\otimes k}\otimes \cup).$$
From the first equation we find
\begin{equation}\label{eqcap0}(\II^{\otimes k}\otimes \cap)\gamma^{(k+2)}_i\gamma^{(k+1)}_j=0,\qquad\mbox{unless}\quad j=i+1. \end{equation}
Equation~\eqref{eqsum1} then further implies
\begin{equation}\label{eqcap1}(\II^{\otimes k}\otimes \cap)\gamma^{(k+2)}_{i}=(\II^{\otimes k}\otimes \cap)\gamma^{(k+2)}_{i}\gamma^{(k+1)}_{i+1}=(\II^{\otimes k}\otimes \cap)\gamma^{(k+1)}_{i+1}.\end{equation}
These equations, and their analogues for~$\cup$ can be used to prove the proposed equalities. We do this explicitly for the first one.

For~$X=(r,e)$, we calculate, using equation~\eqref{eqDefT} and the definitions in \ref{Defepsilon} and \ref{Defiotapi}, that
$$(\varepsilon\circ\iota^{i}_{\TT})_X\;=\; (e\otimes \cap)\circ (e\otimes\II\otimes\II)\gamma_i^{(r+2)} \;=\;(e\otimes \cap)\gamma_i^{(r+2)}$$
and
\begin{eqnarray*}(\varepsilon\circ (\iota^{i}\star (\iota^{i+1}\circ\pi^{i+1})))_X&=&\varepsilon_X\circ \iota^i_{\TT X}\circ\TT_i(\iota_X^{i+1}\circ\pi^{i+1}_X)\\
&=&(e\otimes\cap)\circ (e\otimes\II\otimes\II)\gamma_i^{(r+2)}\circ ((e\otimes\II)\gamma_{i+1}^{(r+1)}\otimes \II) \gamma_i^{(r+2)}\\
&=& (e\otimes \cap)\gamma_i^{(r+2)}\gamma^{(r+1)}_{i+1}.
\end{eqnarray*}
By \eqref{eqcap1}, these two morphisms are the same indeed.
\end{proof}

The following two lemmata were inspired by \cite[Lemmata~2.7 and~2.8]{ES}.
\begin{lemma}\label{Lempsi}
For~$\psi_r:=s_{r\,\minus\,1}(x_{r\,\minus\,1}-x_{r})+1\in A_{r} $, with~$r\ge 2$, we have
\begin{enumerate}[(i)]
\item $x_{r}\psi_r=\psi_rx_{r\,\minus\,1}-\epsilon_{r\,\minus\,1}$,
\item $x_{r\,\minus\,1}\psi_r=\psi_rx_{r}-\epsilon_{r\,\minus\,1}$,
\item $\psi_r^2=1-(x_{r\,\minus\,1}-x_{r})^2,$
\item $\psi_r\circ(a\otimes\II\otimes \II)\;=\;(a\otimes\II\otimes\II)\circ\psi_{s},$ for any $a\in\Hom_{\cA}(s\,\minus\,2,r\,\minus\,2)$.
\end{enumerate}
\end{lemma}
\begin{proof}
Parts (i)-(iii) are direct applications of the commutation relations in~\cite[Lemma~6.3.1]{PB1}. Part (iv) follows from the fact that~$s_{r\,\minus\,1}$ is equal to $\II^{\otimes r\,\minus\,2}\otimes \XX$ and Lemma~\ref{LemJM}.
\end{proof}
Now we define a family of even morphisms. For each $X=(r,e)\in\Ob\DA$, we set
$$\varphi_X:\TT\TT(X)\to\TT\TT(X),\qquad\varphi_X=\psi_{r+2}(e\otimes\II\otimes\II)=(e\otimes\II\otimes\II)\psi_{r+2}.$$
Again we extend to arbitrary objects in~$\DA$ and we obtain a natural transformation~$\varphi:\TT\TT\Rightarrow\TT\TT$ by Lemma~\ref{Lempsi}(iv).

\begin{lemma}\label{LemCatb} Take $i,j\in\mZ$, such that~$|i-j|>1$, then
\begin{enumerate}[(i)]
\item $\varphi\circ(\iota^i\star\iota^j)\;=\; (\iota^j\star\iota^j)\circ (\pi^j\star\pi^i)\circ\varphi \circ(\iota^i\star\iota^j)$ as natural transformations $\TT_i\TT_j\Rightarrow \TT\TT$;
\item $(\pi^i\star\pi^j)\circ\varphi\circ\varphi\circ(\iota^i\star\iota^j)$ is a natural isomorphism of~$\TT_i\TT_j$.
\end{enumerate}
\end{lemma}
\begin{proof}
We start by proving part (i). For~$X=(r,e)$, we have
$$(\varphi\circ(\iota^i\star\iota^j))_X\;=\;(e\otimes\II\otimes\II)\psi_{r+2}\gamma_i^{(r+2)}\gamma_j^{(r+1)}.$$
By Lemma~\ref{Lempsi}(i), we have
$$(x_{r+2}-j)(e\otimes\II\otimes\II)\psi_{r+2}\gamma_i^{(r+2)}\gamma_j^{(r+1)}=(e\otimes\II\otimes\II)\left(\psi_{r+2}(x_{r+1}-j)\gamma_i^{(r+2)}\gamma_j^{(r+1)}-\epsilon_{r+1}\gamma_i^{(r+2)}\gamma_j^{(r+1)}\right).$$
As we assume that~$j\not=i+1$, the last term vanishes by equation~\eqref{eqcap0}. Multiplying $(\varphi\circ(\iota^i\star\iota^j))_X$ with~$(x_{r+2}-j)^p$ from the left for the appropriate $p\in\mN$ will thus yield zero, meaning
$$(e\otimes\II\otimes\II)\psi_{r+2}\gamma_i^{(r+2)}\gamma_j^{(r+1)}=\gamma^{(r+2)}_j(e\otimes\II\otimes\II)\psi_{r+2}\gamma_i^{(r+2)}\gamma_j^{(r+1)}.$$
The corresponding reasoning for~$(x_{r+1}-i)$ concludes the proof of part (i).

Now we consider part (ii). By Lemma~\ref{Lempsi}(iii), for~$X=(r,e)$, we have
$$((\pi^i\star\pi^j)\circ\varphi\circ\varphi\circ(\iota^i\star\iota^j))_X\;=\;(e\otimes\II\otimes\II)(1-(x_{r+1}-x_{r+2})^2)\gamma_i^{(r+2)}\gamma_j^{(r+1)}.$$
For any $c\in\mk$, we can expand
$$1-(x_{r+1}-x_{r+2})^2\;=\;(1-c^2)-(x_{r+1}-x_{r+2}-c)^2-2c(x_{r+1}-x_{r+2}-c)$$
If we set $c=j-i$, then this allows to write the above morphism as the sum of~$(1-c^2)1_{\TT_i\TT_j X}$ and a nilpotent one. Since~$c^2\not=1$, this means that the morphism is an isomorphism of~$\TT_i\TT_j X$.
\end{proof}

\begin{proof}[Proof of Theorem~\ref{ThmCat}]
The relation~$\TT_i^2\cong 0$ follows immediately from Theorem~\ref{ThmDecat}.

Now assume that~$|i-j|>1$. The composition
$$(\pi^i\star \pi^j)\circ\varphi\circ(\iota^j\star \iota^i)\circ(\pi^j\star\pi^i)\circ\varphi\circ(\iota^i\star\iota^j)$$
corresponding to
$$\TT_i\TT_j\Rightarrow\TT\TT\Rightarrow\TT\TT\Rightarrow \TT_j\TT_i\Rightarrow\TT\TT\Rightarrow\TT\TT\Rightarrow \TT_i\TT_j$$
is an isomorphism, by Lemma~\ref{LemCatb}. We hence have even natural transformations $\alpha:\TT_i\TT_j\Rightarrow \TT_j\TT_i$ and~$\beta:\TT_j\TT_i\Rightarrow\TT_i\TT_j$ such that~$\beta\circ\alpha$ is an isomorphism. Since~$\TT_i\TT_jX\cong\TT_j\TT_iX$ for all $X\in\Ob\DA$, see Theorem~\ref{ThmDecat}, this means that both $\alpha$ and~$\beta$ must be isomorphisms. 

Now we consider the natural transformation
$$\pi^i\circ\TT(\varepsilon)\circ((\iota^i\circ\pi^i)\star (\iota^{i-1}\circ \pi^{i-1})\star (\iota^i\circ\pi^i))\circ\eta_{\TT}\circ\iota^i,$$
corresponding to
$$\TT_i\Rightarrow \TT\Rightarrow\TT\TT\TT\Rightarrow\TT_i\TT_{i-1}\TT_i\Rightarrow\TT\TT\TT\Rightarrow \TT\Rightarrow\TT_i.$$
Using the standard interchange laws $\eta_{\TT}\circ\iota^i=\TT\TT(\iota^i)\circ\eta_{\TT_i}$ and~$\pi^i\circ\TT(\varepsilon)=\TT_i(\varepsilon)\circ\pi^i_{\TT\TT}$, subsequently Lemma~\ref{LemCat3}, again the interchange laws, and finally Lemma~\ref{LemCat2}, shows that the composition above is equal 
to $1_{\TT_i}=\pi^i\circ\iota^i$.
In particular, we find odd natural transformations $\alpha:\TT_i\Rightarrow \TT_i\TT_{i-1}\TT_i$ and~$\beta: \TT_i\TT_{i-1}\TT_i\Rightarrow\TT_i$ such that~$\beta\circ\alpha=1_{\TT_i}$. As $\TT_i(X)\cong\TT_i\TT_{i-1}\TT_i(X)$ for all $X\in\Ob\DA$, see Theorem~\ref{ThmDecat}, it follows that~$\alpha$ and~$\beta$ are isomorphisms.
The relation for~$i+1$ follows similarly.
\end{proof}

\begin{proof}[Proof of Proposition~\ref{PropAff}]
We define $\tau:\TT\TT\Rightarrow\TT\TT$ as $\eta\circ\varepsilon$. Similarly, we define $\sigma:\TT\TT\Rightarrow\TT\TT$ by setting
$$\sigma_X=(e\otimes \XX)=s_{r+1} (e\otimes \II\otimes \II),\qquad\mbox{for $X=(r,e)$.}$$ 
Now we argue that the relations in \cite[Definition~3.1]{ChenPeng} are satisfied. It is easy to see that it suffices to prove that evaluation on the objects $k\in\Ob\cA\subset\Ob\DA$ 
actually yields morphisms
$$\widehat{P}^-_d\to\End_{\DA}(T^d(k))=\End_{\cA}(d+k)=A_{d+k}.$$
That we indeed get an algebra morphism
$$\widehat{P}^-_d\to A_{d+k},\qquad s_i\mapsto s_{i+k},\; \epsilon_i\mapsto \epsilon_{i+k}, \;y_j\mapsto x_{j+k},$$
then follows immediately from consistency between the relations in \cite[Definition~3.1]{ChenPeng} and \cite[Section~6.3]{PB1}.
\end{proof}

\subsection{Relation with other categorical representations}
By~\cite[Theorem~4.5.1]{gang}, the functors $\Theta_j$ on~$\Fn$ yield a categorical representation of~$\TL_{\infty}(0)$ on~$\Fn$.
That result served as inspiration for the statement in Theorem~\ref{ThmCat}. Both categorical representations are actually intimately connected, despite the fact that one is on an abelian and one on an additive category. We briefly explore the relation in this section.

\subsubsection{}\label{Compare1}By Lemma~\ref{LemTTTheta}, the decategorification of~$F_n$ is a morphism of~$\TL_{\infty}(0)$-modules. Since~$F_n$ has a kernel, this is not a monomorphism. Moreover, $F_n$ is not essentially surjective and more importantly it is not clear whether the induced morphism
$$[\DA]_{\oplus}\;\to\; [\Fn]$$
from the split Grothendieck group of~$\DA$ to the Grothendieck group of~$\Fn$ is surjective.

\subsubsection{}Since~$\Fn$ has infinite global dimension, the canonical group monomorphism
$$[\pe(n)\mbox{-proj}]_{\oplus}\;\hookrightarrow\;[\Fn]$$
will not be an isomorphism. In fact, by the results in Section~\ref{SecTenProj}, the functor~$\Theta$ restricts to the full additive subcategory~$\pe(n)$-proj and this subcategory constitutes the socle of the categorical $\TL_{\infty}(0)$-representation on~$\Fn$ in~\cite[Theorem~4.5.1]{gang}.

By Theorem~\ref{ThmProj2}, we have an essentially bijective (and full) $\mk$-linear functor from $\ideal_n/\ideal_{n+1}$ to $\pe(n)$-proj, so in particular
$$[\ideal_n/\ideal_{n+1}]_{\oplus} \;\cong\;[\pe(n)\mbox{-proj}]_{\oplus}.$$
By construction and~\ref{Compare1}, this is an isomorphism between the~$\TL_{\infty}(0)$-representation~$\Xi^n$ in \eqref{eqsections} and the decategorification of the socle of the categorical representation of \cite[Theorem~4.5.1]{gang} on~$\Fn$.

In conclusion, our categorical representation~$\Xi$ of~$\TL_{\infty}(0)$ on~$\DA$ admits a filtration, labelled by~$n\in\mN$, where each composition factor `corresponds to' a categorical representation of~$\TL_{\infty}(0)$ on the category of projective modules over $\mathfrak{pe}(n)$ introduced in~\cite{gang}.

\subsubsection{} Contrary to the representation~$\Xi$, the decategorification of \cite[Theorem~4.5.1]{gang} for a fixed $\pe(n)$ is not a faithful representation of~$\TL_{\infty}(0)$. Indeed, using the combinatorics of \cite[Section~5.2]{gang}, it follows that a generic functor of the form
$$\Theta_{i_1}\Theta_{i_2}\cdots\Theta_{i_p}$$
with~$p>n$ will send any (thick) Kac module to zero. Since the functor is exact this automatically implies that it maps any module to zero. In particular $T_{i_1}T_{i_2}\cdots T_{i_p}$ will generically act as zero on the decategorified representation of~$\TL_{\infty}(0)$.

\subsection*{Acknowledgement}
We thank Volodymyr Mazorchuk, Catharina Stroppel and Oded Yacobi for very useful discussions. We would also like to thank Joanna Meinel for useful comments. This research was supported by the Australian Research Council grants DP140103239 and DP150103431.

\end{document}